\documentclass[final,hidelinks,onefignum,onetabnum]{siamart220329}

\usepackage{amscd,amssymb, amsmath,amsfonts, wasysym, mathrsfs, enumerate, mathtools,hhline,xcolor}%
\usepackage{graphicx}
\usepackage[all, cmtip]{xy}
\usepackage{multirow}

\definecolor{hot}{RGB}{65,105,225}
\usepackage{algorithm}
\usepackage[noend]{algpseudocode}

\usepackage{hyperref}
\hypersetup{
	plainpages=false,
    colorlinks,
    linktocpage=true,
    linkcolor=hot,
    citecolor=hot,
    urlcolor=hot
}
\usepackage[hyperpageref]{backref}

\usepackage[colorinlistoftodos,prependcaption,textsize=footnotesize]{todonotes}

\makeatletter
\@mparswitchfalse%
\makeatother
\normalmarginpar %

\newtheorem{remark}[theorem]{Remark}
\newtheorem{assumption}[theorem]{Assumption}

\begin{document}

\title{Domain decomposition methods with Physics-informed neural networks for elliptic equations on manifolds}

\author{Yufang Jiang \thanks{School of Mathematical Sciences, Nanjing Normal University, Nanjing Jiangsu, China (\email{240902008@njnu.edu.cn}).}
\and Lizhen Qin \thanks{School of Mathematics, Nanjing University, Nanjing, Jiangsu, China
  (\email{qinlz@nju.edu.cn}).}
\and Feng Wang \thanks{School of Mathematical Sciences, Nanjing Normal University, Nanjing Jiangsu, China (\email{fwang@njnu.edu.cn}).}
}

\maketitle

\headers{DDMs with PINNs on Manifolds}{Y. Jiang, L. Qin, and F. Wang}

\begin{abstract}
We propose two numerical domain decomposition methods (DDMs) for elliptic equations on compact Riemannian manifolds, based on physics-informed neural networks (PINNs). Our approach incorporates the DDM technique for manifolds with the advantages of neural networks in high-dimensional settings. The proposed methods are validated through numerical experiments on various manifolds, both with and without boundary, in dimensions ranging from $5$ to $10$.
\end{abstract}

\begin{AMS}{Primary 65N55; Secondary 58J05, 68T07.}
\end{AMS}


\section{Introduction}\label{sec_introduction}
Elliptic equations on Riemannian manifolds are important in both analysis and geometry (see e.g.~\cite{Aubin,jost,schoen_yau}).
These equations appear in many areas, such as image processing, multifluid dynamics, micromagnetics, and theoretical physics (see e.g.~\cite{bachini2021intrinsic,Barrera_Kolev_2023,bonito2020divergence,dobrev2010surface,Holst2018,Holst2016Wave,Holst_Stern,jankuhn2021error,jin2021gradient,mohamed2005finite,reuter2009discrete}). A simple and important example is
\begin{equation}
\label{eqn_laplace_problem}
- \Delta u + bu = f.
\end{equation}
Here $\Delta$ is the Laplace-Beltrami operator, or Laplacian for brevity, defined on a $d$-dimensional Riemannian manifold $M$. Many manifolds of interest arise naturally as submanifolds of Euclidean spaces, where the ``dimension'' of a submanifold is understood as the topological dimension of the manifold, not that of the ambient Euclidean space. For instance, the unit sphere
\[
S^{n-1} = \{ x \in \mathbb{R}^{n} \mid \| x \| =1 \}
\]
in $\mathbb{R}^{n}$ has dimension $n-1$.

When the manifold $M$ is a two-dimensional Riemannian submanifold in $\mathbb{R}^{3}$, i.e. a surface, the numerical methods to solve differential equations, particularly (\ref{eqn_laplace_problem}), on $M$ have been extensively studied over many years (see e.g.~\cite{baumgardner_frederickson,dziuk88,dziuk91,nedelec,nedelec_planchard}). A conventional and popular approach is to solve the equations by finite element methods (FEMs) based on a global grid or triangulation of $M$. Such a grid can be obtained by polyhedron approximation in $\mathbb{R}^{3}$. This approach has been highly developed and widely applied (see e.g.~\cite{BDN20,DDE05,dziuk_elliott} for surveys and bibliographies).

However, as pointed out in \cite{cao_qin} and \cite{qin_wang_wang}, the above approach on surfaces faces substantial obstacles when extended to general high-dimensional manifolds, due to their great topological and geometrical complexity. Moreover, many important manifolds are even not defined as submanifolds of Euclidean spaces. A case in point is the complex projective space $\mathbb{CP}^{k}$ which is foundational to algebraic geometry. Such manifolds pose formidable challenges in numerical computation, particularly in the construction of global triangulations.

To circumvent this difficulty, Qin-Zhang-Zhang in \cite{qin_zhang_zhang} proposed an idea to solve elliptic equations on manifolds using FEMs without global triangulations. Since a $d$-dimensional manifold $M$ has local coordinate charts by definition, $M$ can be decomposed into finitely many subdomains that carry local coordinates. Consequently, an elliptic equation on each subdomain can be transformed to one on a domain in $\mathbb{R}^{d}$. Thus an elliptic problem on $M$ can be solved by domain decomposition methods (DDMs) with subproblems posed on Euclidean domains, where grids are much easier to generate.

The idea of \cite{qin_zhang_zhang} was further developed by Cao-Qin in \cite{cao_qin} and by Qin-Wang-Wang in \cite{qin_wang_wang}. Those works merge the philosophy in \cite{qin_zhang_zhang} and the seminal works of P.~L.~Lions. To solve elliptic differential equations in Euclidean domains, Lions proposed a sequential DDM in \cite[Section~I.~4]{lions1} and a parallel DDM in \cite[p.~66]{lions2}. Both methods are overlapping DDMs formulated at continuous level, i.e.~they are not numerical methods. These two DDMs were adapted and generalized to manifolds in \cite{cao_qin} and \cite{qin_wang_wang}, respectively. Furthermore, \cite{cao_qin} and \cite{qin_wang_wang} also proposed DDMs based on FEMs on manifolds which are discrete analogues of the aforementioned continuous DDMs. Numerical experiments demonstrated that these numerical DDMs perform well on various $4$-dimensional manifolds with and without boundary. The present work further extends these previous studies by incorporating neural network techniques.

In fact, as the dimension increases, FEMs face significant challenges. Due to the curse of dimensionality, for a fixed mesh size, the size of the stiffness matrix grows exponentially with respect to the dimension (cf.~\cite{ciarlet}). Though the numerical methods in \cite{cao_qin} and \cite{qin_wang_wang} are theoretically applicable in any dimension given sufficient computing resources, they are only feasible in practice when the dimension is not too high.

To tackle problems on high-dimensional manifolds, in this paper, we combine the continuous DDMs in \cite{cao_qin} and \cite{qin_wang_wang} with the method of physics-informed neural networks (PINNs). Following \cite{cao_qin} and \cite{qin_wang_wang}, we convert a global elliptic problem on a general $d$-dimensional manifold $M$ into subproblems on a collection of domains in $\mathbb{R}^{d}$. Each subproblem is then solved using the PINNs method. In other words, we propose two overlapping DDMs, one sequential and one parallel, that employ PINNs at the numerical level for solving elliptic problems on manifolds. The PINNs method, popularized by \cite{RPK2019}, offers many advantages; we primarily exploit its ability to mitigate the curse of dimensionality. We test our methods on various manifolds, with and without boundary, in dimensions raging from $5$ to $10$. Numerical results indicate that these DDMs perform well.

Neural networks also have the merit of being mesh-free. Nonetheless, even though the difficulty of grid construction is eliminated, the framework of DDMs remains valuable for numerical computations on manifolds because coordinates are still indispensable for any numerical implementation. Since a general manifold $M$ does not necessarily admit a global coordinate chart, it's difficult to construct a global neural network on a complicated $M$, especially when $M$ is not defined as a submanifold of $\mathbb{R}^{n}$. A more practical strategy, following \cite{qin_zhang_zhang}, is to cover $M$ by a collection of subdomains $M_{i}$, each lying within a local coordinate chart. The global problem on $M$ can then be solved numerically using these local coordinate systems.

DDMs have a very long history, beginning with the alternating method invented by H.~A.~Schwarz \cite{schwarz}. The seminal works \cite{lions1,lions2,lions3} of P.~L.~Lions are natural and remarkable extensions of Schwarz's original idea. Since then, the field has grown significantly (see e.g.~\cite{dolean2015DDM,quarteroni_valli,smith_bjorstad_gropp,toselli_widlund}). In many later developments, DDM schemes, both overlapping and nonoverlapping, are used as preconditioners for globally discretized problems (see e.g.~\cite{AGJT2021,BPWX,1996CaiSaadOverlapping,dryja2013,xu92,xu_zou}); within this context, Lions' sequential and parallel DDMs are related to the multiplicative and additive Schwarz methods, respectively. Neural networks have also been under active development for decades (see e.g.~\cite{Barron,cybenko,Dissanayake_Phan-Thien,LLF1998,LLPS1993,Psichogios_Ungar}). In recent years, the field has seen rapid progress, and neural networks have emerged as a powerful and versatile numerical tool with broad impact across computational mathematics (see e.g.~\cite{Ainsworth_Dong,DeRyck_Mishra,E2017,EHJ2017,E_Yu2018,GYZhou2023,GWYZhou2022,HJE2018,HLXZ2020,HJKK2022,Li_Chen2026,LKALBSA2021,LLMDong2019,LZCChen2022,PLK2019,QKPW2020,RPK2019,WLLLXie2024,xu2020,ZBYZ2020}). For the numerical solution of differential equations (linear or nonlinear) on Euclidean domains, numerous studies have combined DDM frameworks with neural network approaches (see e.g.~\cite{DHMM2024,HKLW2021,KKKK2024,LKK2026,LXLL2023,Li_Xiang_Xu,SHMWang2025,sun_xu_yi_2024,sun_xu_yi_2025,TBCP2024}).

It is worth emphasizing that our work differs substantially from the earlier works on Euclidean-domain problems mentioned above. First, we apply DDMs to general high-dimensional manifolds. Particularly, we avoid constructing a global neural network on the target manifold. Instead of employing preconditioner techniques that rely on global networks, we follow more closely Lions' original approach, which solves local subproblems exclusively. Furthermore, in our methods, data exchange between subdomains is carried out via coordinate transition functions, reflecting a fundamental feature of manifolds.

The outline of this paper is as follows. In Section \ref{sec_continuous}, we introduce the model problem and recall the continuous-level DDMs proposed in \cite{cao_qin} and \cite{qin_wang_wang}. In Section \ref{sec_scheme}, we propose numerical DDMs based on PINNs which are numerical analogues of the continuous ones. Finally, some numerical results are reported in Section \ref{sec_experiment}.

\section{Theory on Continuous Problems}\label{sec_continuous}
In this section, we shall first formulate an elliptic model problem of second order on manifolds. Then we recall the domain decomposition methods (DDMs) developed in \cite{cao_qin} and \cite{qin_wang_wang} to solve the model problem (see Algorithms \ref{alg_continuous_sequential} and \ref{alg_continuous_parallel} below). These DDMs originated from the seminal works \cite{lions1,lions2} of P.~L.~Lions to solve differential equations in domains in Euclidean spaces. The iterative procedures also motivate us to propose numerical procedures based on physics-informed neural networks (PINNs) (Algorithms \ref{alg_numerical_sequential} and \ref{alg_numerical_parallel}) in $\S$\ref{sec_scheme}, which provide numerical approximations to the solution to the model problem.

\subsection{Model Problem}
Let $M$ be a $d$-dimensional compact smooth manifold without or with boundary $\partial M$. Equip $M$ with a Riemannian metric $g$. With respect to $g$, the Laplace operator $\Delta$, also named the Laplace-Beltrami operator, is defined on $M$. Note that neither $g$ nor $\Delta$ can be expressed by coordinates globally in general because $M$ does not necessarily have a global coordinate chart. In a local chart with coordinates $(x_{1}, \dots, x_{d})$, the Riemannian metric tensor $g$ is expressed as
\begin{equation}\label{eqn_metric}
g= \sum_{\alpha, \beta=1}^{d} g_{\alpha \beta} \mathrm{d} x_{\alpha} \otimes \mathrm{d} x_{\beta},
\end{equation}
where the matrix $(g_{\alpha \beta})_{d \times d}$ is symmetric and positive definite. The Laplace operator $\Delta$ is expressed as
\[
\Delta u = \frac{1}{\sqrt{G}}\sum_{\alpha=1}^{d}\frac{\partial}{\partial x_{\alpha}} \left( \sum_{\beta=1}^{d}g^{\alpha \beta} \sqrt{G} \frac{\partial u}{\partial x_{\beta}} \right),
\]
where $G = \det \left( (g_{\alpha \beta})_{d \times d} \right)$ is the determinant of the matrix $(g_{\alpha \beta})_{d \times d}$ and $(g^{\alpha \beta})_{d \times d}$ is the inverse of $(g_{\alpha \beta})_{d \times d}$. It's well-known that $\Delta$ is an elliptic differential operator of second order.

We consider the following model problem on $M$
\begin{equation}\label{eqn_problem}
\left\{
\begin{aligned}
Lu= - \Delta u + bu & = f, \\
u|_{\partial M} & = \varphi,
\end{aligned}
\right.
\end{equation}
where $b \geq 0$ is a constant. It suffices to solve \eqref{eqn_problem} on each component of $M$. Therefore, without loss of generality, $M$ is assumed to be connected. Note that, if $\partial M = \emptyset$, the boundary condition $u|_{\partial M} = \varphi$ is vacuously satisfied. However, to guarantee that \eqref{eqn_problem} is well-posed, we need to further assume $b>0$ if $\partial M = \emptyset$.

In a local chart, the operator $L$ in \eqref{eqn_problem} can be expressed via coordinates as
\begin{align}\label{eqn_differential_operator}
    L u & = -\frac{1}{\sqrt{G}} \sum_{\alpha=1}^{d} \frac{\partial}{\partial x_{\alpha}} \left( \sum_{\beta=1}^{d} g^{\alpha \beta} \sqrt{G} \frac{\partial u}{\partial x_{\beta}} \right) +bu \nonumber \\
    & = -\sum_{\alpha,\beta=1}^{d} g^{\alpha \beta}  \frac{\partial^{2} u}{\partial x_{\alpha} \partial x_{\beta}} - \frac{1}{\sqrt{G}} \sum_{\beta=1}^{d} \sum_{\alpha=1}^{d} \frac{\partial}{\partial x_{\alpha}} \left( g^{\alpha \beta} \sqrt{G} \right) \frac{\partial u}{\partial x_{\beta}} +bu.
\end{align}
This is an elliptic operator with $C^{\infty}$ coefficients and is in a quite general form since the Riemannian metric $g$ can be arbitrary.

We assume that the problem \eqref{eqn_problem} is well-posed, i.e.~it has a unique solution $u$. The well-posedness is granted by the ellipticity of \eqref{eqn_problem} provided that both $f$ and $\varphi$ have a certain mild regularity (cf.~\cite[Chapters~6~\&~8]{Gilbarg_Trudinger} and \cite[Chapter~3]{Aubin}).

\subsection{Continuous Algorithms}\label{subsec_continuous_algorithm}
We describe domain decomposition methods (DDMs) to solve (\ref{eqn_problem}). Suppose $M$ is decomposed into $m$ subdomains with $m>1$, i.e.~$M = \bigcup_{i=1}^{m} M_{i}$. Here $M_{i}$ is a closed subdomain (submanifold with codimension $0$) of $M$. Let $\gamma_{i} = \partial M_{i} \setminus \partial M$. We see $\overline{\gamma_{i}} \subseteq \partial M_{i}$, and $M_{i} \setminus \overline{\gamma_{i}}$ is the interior of $M_{i}$ in the sense of point set topology. If $M_{i} \cap \partial M = \emptyset$, particularly $\partial M = \emptyset$, then $\overline{\gamma_{i}} = \gamma_{i} = \partial M_{i}$. We make the following assumption on decompositions.

\begin{assumption}\label{asp_decomposition}
Suppose $M = \bigcup_{i=1}^{m} (M_{i} \setminus \overline{\gamma_{i}})$, and each $\partial M_{i}$ is Lipschitz.
\end{assumption}

By Assumption \ref{asp_decomposition}, these $M_{i} \setminus \overline{\gamma_{i}}$ form an open cover of $M$. So our decomposition is overlapping. We firstly formulate a sequential DDM, i.e.~Algorithm \ref{alg_continuous_sequential}. It was originally proposed by P.~L.~Lions in \cite[I.~4]{lions1} to solve equations in Euclidean domains, in which the classical Schwarz Alternating Method is extended to an iterative procedure with many subdomains. Later, this method was generalized by \cite[Algorithm~2.1]{cao_qin} to solve equations on manifolds.

\begin{algorithm}[ht]
\caption{(\cite{lions1,cao_qin})~Sequential DDM to solve \eqref{eqn_problem}, under Assumption \ref{asp_decomposition}.}
\label{alg_continuous_sequential}

\begin{algorithmic}[1]
\State%
Choose an initial guess function $u^{0}$ on $M$ with $u^{0}|_{\partial M} = \varphi$.

\State%
For each $n>0$,
\begin{quote}
define $u^{n}_{0} = u^{n-1}$.

For $1 \leq i \leq m$,
\begin{quote}
(assuming $u^{n}_{j}$ has been obtained for all $j<i$,) find a function $u^{n}_{i}$ on $M$ such that
\[
\left\{
\begin{aligned}
L u^{n}_{i} & = f, & \text{in $M_{i} \setminus \partial M_{i}$;}
\\
u^{n}_{i} & = u^{n}_{i-1}, & \text{otherwise.}
\end{aligned}
\right.
\]
\end{quote}
\end{quote}

Let $u^{n} = u^{n}_{m}$.

\end{algorithmic}
\end{algorithm}

Lions described Algorithm \ref{alg_continuous_sequential} as ``sequential" (see the sentence right before (26) in \cite{lions1}), which means the subproblems on subdomains $M_{i}$ have to be solved one by one sequentially. On the other hand, Lions later invented in \cite[p.~66]{lions2} another DDM to solve equations in Euclidean domains. The later one was claimed as ``parallel" (see a few sentences before (37) in \cite{lions2}). This parallel procedure was adapted and generalized to solve equations on manifolds in \cite[Algorithm~2.1]{qin_wang_wang}.

The adaption in \cite{qin_wang_wang} needs a partition of unity $\{ \rho_{i} \mid 1 \leq i \leq m \}$ on $M$. In other words, $\rho_{i}$'s are a family of nonnegative functions on $M$ and $\sum_{i=1}^{m} \rho_{i} = 1$.

\begin{assumption}\label{asp_partition}
The partition of unity  $\{ \rho_{i} \mid 1 \leq i \leq m \}$ is subordinate to the decomposition in Assumption \ref{asp_decomposition}, i.e.~$\mathrm{supp} \rho_{i} \subset M_{i} \setminus \overline{\gamma_{i}}$, where $\mathrm{supp} \rho_{i}$ is the support of $\rho_{i}$. Furthermore, each $\rho_{i}$ is continuous.
\end{assumption}

Let's recall the parallel DDM in \cite{qin_wang_wang}, i.e.~Algorithm \ref{alg_continuous_parallel}.
\begin{algorithm}
\caption{(\cite{qin_wang_wang})~Parallel DDM to solve \eqref{eqn_problem}, under Assumptions \ref{asp_decomposition} and \ref{asp_partition}.}
\label{alg_continuous_parallel}

\begin{algorithmic}[1]
\State%
Choose an initial guess function $u^{0}$ on $M$ with $u^{0}|_{\partial M} = \varphi$.

\State%
For each $n>0$,
\begin{quote}
for $1 \leq i \leq m$,
\begin{quote}
find a function $u^{n}_{i}$ on $M_{i}$ such that
\[
\left\{
\begin{aligned}
L u^{n}_{i}  & = f, & \text{in $M_{i} \setminus \partial M_{i}$;}
\\
u^{n}_{i}  & = u^{n-1}, & \text{otherwise.}
\end{aligned}
\right.
\]
\end{quote}
\end{quote}

Let $u^{n} = \sum_{i=1}^{m} \rho_{i} u^{n}_{i}$.
\end{algorithmic}
\end{algorithm}

There are two obvious differences between the above two algorithms. First, unlike Algorithm \ref{alg_continuous_sequential}, the parallel Algorithm \ref{alg_continuous_parallel} solves the subproblems on subdomains $M_{i}$ independently and parallelly at each iterative step $n$. Second, each $u^{n}_{i}$ in Algorithm \ref{alg_continuous_sequential} is globally defined on $M$, but each $u^{n}_{i}$ in Algorithm \ref{alg_continuous_parallel} is only defined on $M_{i}$. Note that the definition of $u^{n}$ in Algorithm \ref{alg_continuous_parallel} does make sense because $\mathrm{supp} \rho_{i} \subset M_{i}$ for each $i$. (More precisely, the function $\rho_{i} u^{n}_{i}$ on $M_{i}$ is extended as $0$ outside of $M_{i}$.)

Algorithm \ref{alg_continuous_parallel} certainly follows \cite[p.~66]{lions2} in large. However, it is necessary to point out that, even in the special case where the manifold $M$ is a Euclidean domain, Algorithm \ref{alg_continuous_parallel} is different from \cite{lions2}. Actually, Algorithm \ref{alg_continuous_parallel} employs a partition of unity which was not utilized in \cite{lions2}. (For a detailed description of this essential difference, see \cite[p.~A920]{qin_wang_wang}.) Due to the difference, it has been proved in \cite{qin_wang_wang} that Algorithm \ref{alg_continuous_parallel} has a better convergence theory (see \cite[Remark~2.5]{qin_wang_wang}).

If $M$ is decomposed into two overlapping subdomains, i.e.~$M = \bigcup_{i=1}^{2} (M_{i} \setminus \overline{\gamma_{i}})$, Algorithms \ref{alg_continuous_sequential} and \ref{alg_continuous_parallel} are both reduced to the famous Schwarz Alternating Method which was invented by H.~A.~Schwarz (\cite{schwarz}) to solve differential equations in Euclidean domains. The following Algorithm \ref{alg_continuous_alternating} is the Schwarz Alternating Method on manifolds.

\begin{algorithm}[ht]
\caption{Schwarz Alternating Method to solve \eqref{eqn_problem}, under Assumption \ref{asp_decomposition} with $m=2$.}
\label{alg_continuous_alternating}

\begin{algorithmic}[1]
\State%
Choose an initial guess functions $u^{0}_{2}$ on $M_{2}$ with $u^{0}_{2}|_{\partial M} = \varphi$.

\State%
For each $n>0$,
\begin{quote}
find a function $u^{n}_{1}$ on $M_{1}$ such that
\[
\left\{
\begin{aligned}
L u^{n}_{1} & = f, & \text{in $M_{1} \setminus \partial M_{1}$;}
\\
u^{n}_{1} & = u^{n-1}_{2}, & \text{on $\gamma_{1}$;}
\\
u^{n}_{1} & = \varphi, & \text{otherwise.}
\end{aligned}
\right.
\]
Find a function $u^{n}_{2}$ on $M_{2}$ such that
\[
\left\{
\begin{aligned}
L u^{n}_{2} & = f, & \text{in $M_{2} \setminus \partial M_{2}$;}
\\
u^{n}_{2} & = u^{n}_{1}, & \text{on $\gamma_{2}$;}
\\
u^{n}_{2} & = \varphi, & \text{otherwise.}
\end{aligned}
\right.
\]
\end{quote}
\end{algorithmic}
\end{algorithm}

In Algorithm \ref{alg_continuous_alternating}, one can extend $u^{n}_{1}$ and $u^{n}_{2}$ to be functions on $M$ by defining $u^{n}_{1}|_{M \setminus M_{1}} = u^{n-1}_{2}$ and $u^{n}_{2}|_{M \setminus M_{2}} = u^{n}_{1}$. Therefore, Algorithm \ref{alg_continuous_alternating} is clearly a special case of Algorithm \ref{alg_continuous_sequential} since the alternating procedure is sequential in nature.

We point out that, Algorithm \ref{alg_continuous_parallel} also degenerates to Algorithm \ref{alg_continuous_alternating} for a two-subdomains decomposition. First of all, now a partition of unity is superfluous to the updating of boundary values. We necessarily have $\rho_{i}|_{\overline{\gamma_{i}}} \equiv 0$ and $\rho_{i}|_{\overline{\gamma_{3-i}}} \equiv 1$ for $i=1,2$. In Algorithm \ref{alg_continuous_parallel}, therefore, $u^{n}|_{\gamma_{1}} = u^{n}_{2}$ and $u^{n}|_{\gamma_{2}} = u^{n}_{1}$. Second, setting $u^{0}_{i} = u^{0}|_{M_{i}}$, Algorithm \ref{alg_continuous_parallel} then splits into two independent iterative processes:
\[
u^{0}_{2}, u^{1}_{1}, u^{2}_{2}, u^{3}_{1}, u^{4}_{2}, u^{5}_{1}, \dots
\]
and
\[
u^{0}_{1}, u^{1}_{2}, u^{2}_{1}, u^{3}_{2}, u^{4}_{1}, u^{5}_{2}, \dots
\]
Each process is a Schwarz alternating procedure. Either one of the two is sufficient for the purpose of approximation.

\subsection{Convergence Theory}
Algorithms \ref{alg_continuous_sequential} and \ref{alg_continuous_parallel} converge when $f$, $\varphi$, $u$ and $u^{0}$ have some mild regularities.

When $M$ is a Euclidean domain and $\varphi =0$ in \eqref{eqn_problem}, the $H^{1}$-convergence of Algorithm \ref{alg_continuous_sequential} had been proved by P.~L.~Lions in \cite[Theorem.~I.2]{lions1}. The proof was generalized to the case of manifolds in \cite[Theorem~2.1]{cao_qin}.
\begin{theorem}[\cite{lions1,cao_qin}~Convergence of Algorithm \ref{alg_continuous_sequential}]\label{thm_sequential_converge}
Under Assumption \ref{asp_decomposition}, suppose further $f \in H^{-1} (M)$ and $\varphi \in H^{\frac{1}{2}} (\partial M)$ in \eqref{eqn_problem}. Suppose $u^{0} \in H^{1} (M)$ in Algorithm \ref{alg_continuous_sequential}. Let $u^{n}$ be the one in Algorithm \ref{alg_continuous_sequential}. Then the following holds:
\begin{enumerate}[(1)]
\item For each $n>0$, $u^{n} \in H^{1} (M)$.

\item There exist constants $C_{0} >0$ and $C \in [0,1)$ such that
\[
\| u- u^{n} \|_{H^{1} (M)} \leq C_{0} C^{n} \| u- u^{0} \|_{H^{1} (M)},
\]
where $C_{0}$ and $C$ are independent of $u$, $u^{0}$ and $n$.
\end{enumerate}
\end{theorem}
\begin{proof}
This has been proved in \cite[Theorem~2.1]{cao_qin} when $\varphi =0$. In general, since $\varphi \in H^{\frac{1}{2}} (\partial M)$, there exists a $v \in H^{1} (M)$ such that $v|_{\partial M} = \varphi$. Now $u-v$ solves the following problem with homogeneous boundary value
\begin{equation}\label{thm_sequential_converge_1}
\left\{
\begin{aligned}
L(u-v) & = f-Lv, \\
(u-v)|_{\partial M} & = 0,
\end{aligned}
\right.
\end{equation}
where $f-Lv \in H^{-1} (M)$. Applying Algorithm \ref{alg_continuous_sequential} to \eqref{thm_sequential_converge_1}, we obtain an iterative sequence $\{ u^{n} - v \}_{n=1}^{\infty}$. By \cite{cao_qin}, the desired conclusion holds for $\{ u^{n} - v \}_{n=1}^{\infty}$. Since
\[
u^{n} = (u^{n} -v) + v \quad \text{and} \quad  u- u^{n} = (u-v) -(u^{n} -v),
\]
the conclusion of this theorem follows.
\end{proof}

The following theorem was proved in \cite[Theorem~2.4]{qin_wang_wang}.
\begin{theorem}[\cite{qin_wang_wang}~Convergence of Algorithm \ref{alg_continuous_parallel}]\label{thm_parallel_converge}
Under Assumptions \ref{asp_decomposition} and \ref{asp_partition}, assume further all $\rho_{i}$'s are Lipschitz continuous. Suppose $\bigcup_{i \neq j} (\gamma_{i} \cap \gamma_{j})$ has a finite $(d-2)$-dimensional Hausdorff measure. Suppose $f \in H^{-1} (M)$, $u \in C^{0} (M) \cap H^{1} (M)$, and $\varphi \in C^{0} (\partial M) \cap H^{\frac{1}{2}} (\partial M)$. Suppose $u^{0} \in C^{0} (M) \cap H^{1} (M)$ in Algorithm \ref{alg_continuous_parallel}. In addition, $u^{0}$ and $u$ are assumed Lipschitz continuous if $\partial M \ne \emptyset$. Let $u^{n}$ be the one in Algorithm \ref{alg_continuous_parallel}. Then the following holds:
\begin{enumerate}[(1)]
\item For each $n>0$ and each $i$, $u^{n} \in C^{0} (M) \cap H^{1} (M)$ and $u^{n}_{i} \in C^{0} (M_{i}) \cap H^{1} (M_{i})$.

\item $\lim\limits_{n \rightarrow \infty} \| u^{n}_{i} - u \|_{C^{0}(M_{i})} =0$ and $\lim\limits_{n \rightarrow \infty} \| u^{n}_{i} - u \|_{H^{1}(M_{i})} =0$ for each $i$.

\item $\lim\limits_{n \rightarrow \infty} \| u^{n} - u \|_{C^{0}(M)} =0$ and $\lim\limits_{n \rightarrow \infty} \| u^{n} - u \|_{H^{1}(M)} =0$.
\end{enumerate}
\end{theorem}

\begin{remark}
In \cite{cao_qin} and \cite{qin_wang_wang}, the problem \eqref{eqn_problem} was numerically solved by finite element methods (FEMs). Following the convention of FEMs, the $f$ in \eqref{eqn_problem} was assumed to have slightly higher regularity in \cite{cao_qin} and \cite{qin_wang_wang}, more precisely, $f \in L^{2} (M)$. Typically, FEMs result in a linear system $AX=b$, where $A$ is a stiffness matrix and $b$ comes from $f$. If $f \in L^{2} (M)$, then $b$ can be easily obtained via numerical integration. Nevertheless, the proofs of Theorems \ref{thm_sequential_converge} and \ref{thm_parallel_converge} go through even if we merely assume $f \in H^{-1} (M)$.
\end{remark}

\section{Numerical Scheme}\label{sec_scheme}
In this section, we propose numerical DDM iterative procedures (Algorithms \ref{alg_numerical_sequential} and \ref{alg_numerical_parallel} below) to obtain approximations to the solution of \eqref{eqn_problem}. They are numerical imitations of Algorithms \ref{alg_continuous_sequential} and \ref{alg_continuous_parallel} based on physics-informed neural networks (PINNs). The idea is as follow. First, $M$ is decomposed into overlapping subdomains $M_{i}$ ($1 \leq i \leq m$), and each $M_{i}$ is in a coordinate chart. Theoretically, we can apply Algorithms \ref{alg_continuous_sequential} and \ref{alg_continuous_parallel} by virtue of this decomposition. Second, since $M_{i}$ is in a coordinate chart, each subproblem in $M_{i}$ can be naturally converted to one in a domain $D_{i}$ in a Euclidean space. In practice, we may use PINNs method to solve problems in $D_{i}$ approximately. The combination of these two ingredients yields our numerical algorithms.

\subsection{Domain Decompositions}\label{subsec_decompose}
Let $M$ be a $d$-dimensional compact Riemannian manifold with or without boundary. Suppose $M$ is decomposed as Assumption \ref{asp_decomposition}, i.e.~$M = \bigcup_{i=1}^{m} (M_{i} \setminus \overline{\gamma_{i}})$, and each $\partial M_{i}$ is Lipschitz. We can further require that there is a smooth (i.e.~$C^{\infty}$) diffeomorphism $\phi_{i} \colon D_{i} \rightarrow M_{i}$ for each $i$, where $D_{i}$ is a domain in $\mathbb{R}^{d}$ with simple shape and Lipschitz boundary.

Theoretically, we can always get such triples $(M_{i}, D_{i}, \phi_{i})$'s. By the definition of a manifold, for each $\zeta \in M$, there is an open chart neighborhood $U_{\zeta}$ of $\zeta$, i.e.~there is a smooth diffeomorphism $\phi_{\zeta} \colon \Omega_{\zeta} \rightarrow U_{\zeta}$, where $\Omega_{\zeta}$ is an open subset of $\mathbb{R}^{d}_{+}$. Here $\mathbb{R}^{d}_{+}$ is a closed half Euclidean space
\[
\mathbb{R}^{d}_{+} = \{ (x_{1}, \dots, x_{d}) \in \mathbb{R}^{d} \mid x_{d} \ge 0 \}.
\]
If $\zeta \in M \setminus \partial M$, particularly $\partial M = \emptyset$, we can even require that $\Omega_{\zeta}$ is open in $\mathbb{R}^{d}$. Since $\Omega_{\zeta}$ is open in $\mathbb{R}^{d}_{+}$ and $\phi_{\zeta}^{-1} (\zeta) \in \Omega_{\zeta}$, we can choose a neighborhood $D_{\zeta}$ of $\phi_{\zeta}^{-1} (\zeta)$ in $\mathbb{R}^{d}_{+}$ such that $D_{\zeta} \subseteq \Omega_{\zeta}$ and $D_{\zeta}$ is a compact domain with simple shape and Lipschitz boundary. For instance, $D_{\zeta}$ can be chosen as ($d$-dimensional) cubes, balls and polydisks, among others. This yields a diffeomorphism $\phi_{\zeta} \colon D_{\zeta} \rightarrow M_{\zeta} \subset U_{\zeta}$, where $M_{\zeta}$ is a neighborhood of $\zeta$. The interiors of all such $M_{\zeta}$'s provide an open cover of $M$. Since it is compact, $M$ can be covered by the interiors of finitely many such $M_{\zeta}$'s. These finitely many $(M_\zeta, D_\zeta, \phi_{\zeta})$'s yield a desired decomposition of $M$. In the above, the interior of $M_{\zeta}$ is in the sense of point set topology. More precisely, let $\gamma_{\zeta} = M_{\zeta} \cap (M \setminus \partial M)$, then the interior of $M_{\zeta}$ is $M_{\zeta} \setminus \overline{\gamma_{\zeta}}$. Each $M_{\zeta} \setminus \overline{\gamma_{\zeta}}$ is an open subset of $M$. If $M_{\zeta} \cap \partial M = \emptyset$, particularly $\partial M = \emptyset$, then $\overline{\gamma_{\zeta}} = \partial M_{\zeta}$.

Suppose we have the desired triples $(M_{i}, D_{i}, \phi_{i})$'s for $M$. Since an elliptic problem in $M_{i}$ can be converted to one in $D_{i}$ via $\phi_{i}$, for the implement of Algorithms \ref{alg_continuous_sequential} and \ref{alg_continuous_parallel}, we merely need to solve problems in these $D_{i}$'s. However, to update the boundary values on the $\partial D_{i}$'s, transmission of data among different $D_{i}$'s is needed.

For $i \neq j$, let $D_{ij} = \phi_{i}^{-1} (M_{i} \cap M_{j}) \subseteq D_{i}$ and $D_{ji} = \phi_{j}^{-1} (M_{i} \cap M_{j}) \subseteq D_{j}$. (See Fig.~\ref{fig_decomposition} for an illustration, where the whale shaped object is a manifold $M$, the square and the disk stand for $D_{i}$ and $D_{j}$, respectively, and the shadowed parts of $D_{i}$ and $D_{j}$ are $D_{ij}$ and $D_{ji}$, respectively.) If $M_{i} \cap M_{j} \ne \emptyset$, then
\[
\phi_{j}^{-1} \circ \phi_{i} \colon \ D_{ij} \rightarrow D_{ji}
\]
is a diffeomorphism which is the transition of coordinates on the overlapping between $M_{i}$ and $M_{j}$. Suppose $\phi_{j}^{-1} \circ \phi_{i}$ maps some part of $\partial D_{i}$ into $D_{ji} \subseteq D_{j}$, a desired function value on $D_{j}$ is transported to $\partial D_{i}$ via this transition of coordinates.
\begin{figure}[htbp]
\centering
  \includegraphics[width=0.6\textwidth]{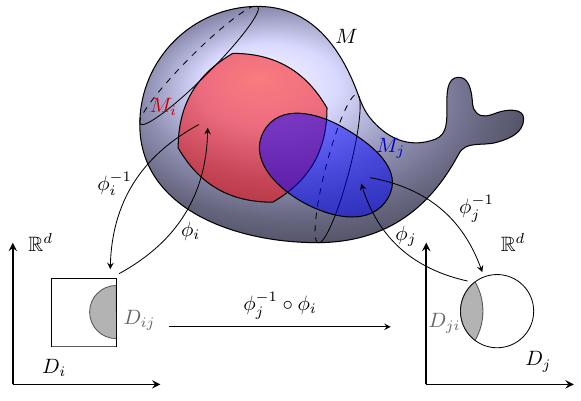}
  \caption{An illustration of domain decompositions.}
  \label{fig_decomposition}
\end{figure}

It's worth to noting that, as pointed in \cite[$\S$4]{cao_qin}, if manifolds $M$ and $M'$ already have desired decompositions, a good decomposition of the product manifold $M \times M'$ is automatically obtained. More precisely, suppose $M$ (resp.~$M'$) has a decomposition $\{ (M_{i}, D_{i}, \phi_{i}) \mid 1 \le i \le m \}$ (resp.~$\{ (M'_{i'}, D'_{i'}, \phi'_{i'}) \mid 1 \le i' \le m' \}$), then $M \times M'$ has the decomposition
\[
\{ (M_{i} \times M'_{i'}, D_{i} \times D'_{i'}, \phi_{i} \times \phi'_{i'}) \mid 1 \le i \le m, 1 \le i' \le m' \}.
\]
The transition of coordinates between $D_{i} \times D'_{i'}$ and $D_{j} \times D'_{j'}$ is
\[
(\phi_{j} \times \phi'_{j'})^{-1} \circ (\phi_{i} \times \phi'_{i'}) = (\phi_{j}^{-1} \circ \phi_{i}) \times ({\phi'}_{j'}^{-1} \circ \phi'_{i'}).
\]
This canonical technique of decomposition for product manifolds will be employed in our numerical test (see $\S$\ref{subsec_B3S3}).

\subsection{Numerical Algorithms}
Now we propose our numerical DDMs to solve \eqref{eqn_problem}. The following assumption is needed.
\begin{assumption}\label{asp_decomposition_numerical}
Suppose $M = \bigcup_{i=1}^{m} (M_{i} \setminus \overline{\gamma_{i}})$. There is a $C^{\infty}$ diffeomorphism $\phi_{i} \colon D_{i} \rightarrow M_{i}$ for each $i$, where $D_{i}$ is a compact domain in $\mathbb{R}^{d}$ with Lipschitz boundary.
\end{assumption}

By \eqref{eqn_differential_operator}, the restriction of \eqref{eqn_problem} to $M_{i}$ is transformed via $\phi_{i}$ to the following equation in $D_{i}$, where $\hat{u}_{i} = u \circ \phi_{i}$.
\begin{equation}\label{eqn_domain_operator}
\hat{L} \hat{u}_{i}:= -\sum_{\alpha,\beta=1}^{d} g^{\alpha \beta}  \frac{\partial^{2} \hat{u}_{i}}{\partial x_{\alpha} \partial x_{\beta}} - \frac{1}{\sqrt{G}} \sum_{\beta=1}^{d} \sum_{\alpha=1}^{d} \frac{\partial}{\partial x_{\alpha}} \left( g^{\alpha \beta} \sqrt{G} \right) \frac{\partial \hat{u}_{i}}{\partial x_{\beta}} +b \hat{u}_{i} = f \circ \phi_{i}.
\end{equation}

We shall solve \eqref{eqn_domain_operator} with suitable boundary values in $D_{i}$ by physics-informed neural networks (PINNs) method. The implementation of the PINNs method requires taking high-order derivatives of the neural networks functions. The activation functions, therefore, should possess $C^{k}$ regularity for sufficiently large $k$. Fixing a $C^{k}$ activation function and a neural network in $D_{i}$, let $V_{i}$ be the function space associated with this neural network. Then $V_{i} \subset C^{k} (D_{i})$. Note that $V_{i}$ is not a linear subspace of $C^{k} (D_{i})$ since it is not closed under addition. Nevertheless, all elements of $V_{i}$ are represented by a fixed finite number, say $N$, of parameters. Therefore, $V_{i}$ is the image of a continuous map from $\mathbb{R}^{N}$ to $C^{k} (D_{i})$.

Define a product space
\[
V := \prod_{i=1}^{m} V_{i} \subset \prod_{i=1}^{m} C^{k} (D_{i}).
\]

We propose the following Algorithm \ref{alg_numerical_sequential} which is a numerical imitation of Algorithm \ref{alg_continuous_sequential}.
\begin{algorithm}[ht]
\caption{Numerical sequential DDM with PINNs to solve \eqref{eqn_problem}, under Assumption \ref{asp_decomposition_numerical}.}
\label{alg_numerical_sequential}

\begin{algorithmic}[1]
\State%
Choose an initial guess $\hat{u}^{0} = (\hat{u}_{1}^{0}, \cdots, \hat{u}_{m}^{0}) \in V$.

\State%
For each $n>0$,
\begin{quote}
for $1 \leq i \leq m$,
\begin{quote}
(assuming $\hat{u}^{n-1}$ and $\hat{u}^{n}_{j}$ have been obtained for all $j<i$,) define a function $\hat{\varphi}^{n}_{i}$ on $\partial D_{i}$ as follows. For $\xi \in \partial D_{i}$, if $\phi_{i} (\xi) \in \partial M$, let $\hat{\varphi}^{n}_{i} (\xi) = \varphi (\phi_{i} (\xi))$. Otherwise, if $\{ j \mid j<i, \phi_{i} (\xi) \in M_{j} \} \neq \emptyset$, then let $j_{0}$ be the maximum of this set and define
\[
\hat{\varphi}^{n}_{i} (\xi) = \hat{u}^{n}_{j_{0}} (\phi_{j_{0}}^{-1} \circ \phi_{i} (\xi)).
\]
Otherwise, let $j_{0} = \max \{ j \mid j \neq i, \phi_{i} (\xi) \in M_{j} \}$ and define
\[
\hat{\varphi}^{n}_{i} (\xi) = \hat{u}^{n-1}_{j_{0}} (\phi_{j_{0}}^{-1} \circ \phi_{i} (\xi)).
\]
Use PINNs method to find a function $\hat{u}^{n}_{i} \in V_{i}$ which approximates the solution to \eqref{eqn_domain_operator} with boundary value $\hat{\varphi}^{n}_{i}$.
\end{quote}
\end{quote}

Define $\hat{u}^{n} = (\hat{u}_{1}^{n}, \cdots, \hat{u}_{m}^{n}) \in V$.

\end{algorithmic}
\end{algorithm}

By Assumption \ref{asp_decomposition_numerical}, we have $\phi_{i} (\partial D_{i}) = \partial M_{i} \subset \partial M \cup \bigcup_{j \ne i} M_{j}$. If $\xi \in \partial D_{i}$ and $\phi_{i} (\xi) \notin \partial M$, then
\[
\{ j \mid j \neq i, \phi_{i} (\xi) \in M_{j} \} \ne \emptyset.
\]
So the $\hat{\varphi}^{n}_{i}$ in Algorithm \ref{alg_numerical_sequential} is well-defined.

Obviously, the $\hat{u}^{n}_{i}$ in Algorithm \ref{alg_numerical_sequential} is a numerical imitation of the $u^{n}_{i}$ in Algorithm \ref{alg_continuous_sequential}. Besides its numerical nature, Algorithm \ref{alg_numerical_sequential} is remarkably different from Algorithm \ref{alg_continuous_sequential} in two aspects. First, the $u^{n}_{i}$ in Algorithm \ref{alg_continuous_sequential} is globally defined on $M$, but the $\hat{u}^{n}_{i}$ in Algorithm \ref{alg_numerical_sequential} is only locally defined on $D_{i}$. Second, the $u^{n}$ in Algorithm \ref{alg_continuous_sequential} is also globally defined on $M$, however, $\hat{u}^{n}$ is a group of local functions. In fact, coordinates are indispensable for numerical computations. Since there are no global coordinates on $M$, it's impossible to express functions concretely over the entire $M$ in practice. To solve \eqref{eqn_problem} numerically, a feasible strategy is to approximate the exact solution $u$ piecewisely on a family of charts $\{ \phi_{i} (D_{i}) \mid 1 \le i \le m \}$ which cover $M$. The approximation of $\hat{u}^{n}_{i}$ to $u \circ \phi_{i}$ is equivalent to the approximation of $\hat{u}^{n}_{i} \circ \phi_{i}^{-1}$ to $u|_{\phi_{i} (D_{i})}$. If $\hat{u}^{n}_{i}$ approximates $u \circ \phi_{i}$ well for each $i$, the goal of numerical computation is achieved. This strategy had also been used in \cite{qin_zhang_zhang,cao_qin,qin_wang_wang} where finite element methods rather than PINNs were applied. In this paper, our numerical experiments in $\S$\ref{sec_experiment} show that this approach also works well for PINNs.

Now we formulate our parallel numerical DDM which needs a partition of unity.
\begin{assumption}\label{asp_partition_numerical}
The partition of unity  $\{ \rho_{i} \mid 1 \leq i \leq m \}$ is subordinate to the decomposition in Assumption \ref{asp_decomposition_numerical}, i.e.~$\mathrm{supp} \rho_{i} \subset M_{i} \setminus \overline{\gamma_{i}}$, where $\mathrm{supp} \rho_{i}$ is the support of $\rho_{i}$, and $M_{i} = \phi_{i} (D_{i})$. Furthermore, each $\rho_{i}$ is continuous.
\end{assumption}

In practice, the $\rho_{i}$'s can be constructed as follows. Choose nonnegative continuous functions $\mu_{i}$'s on $M$ such that $\mathrm{supp} \mu_{i} \subset \phi_{i} (D_{i})  \setminus \overline{\gamma_{i}}$ and $\sum_{i=1}^{m} \mu_{i} >0$ on $M$. A desired $\rho_{i}$ can be defined as
\begin{equation}\label{eqn_rho}
\rho_{i} = \frac{\mu_{i}}{\sum_{j=1}^{m} \mu_{j}}.
\end{equation}
To obtain $\mu_{i}$, it suffices to define $\mu_{i} \circ \phi_{i}$ which is a function on $D_{i}$ and hence can be defined simply and quite arbitrarily in terms of elementary functions.

The following Algorithm \ref{alg_numerical_parallel} is a numerical imitation of Algorithm \ref{alg_continuous_parallel}.
\begin{algorithm}[ht]
\caption{Numerical parallel DDM with PINNs to solve \eqref{eqn_problem}, under Assumptions \ref{asp_decomposition_numerical} and \ref{asp_partition_numerical}.}
\label{alg_numerical_parallel}

\begin{algorithmic}[1]
\State%
Choose an initial guess $\hat{u}^{0} = (\hat{u}_{1}^{0}, \cdots, \hat{u}_{m}^{0}) \in V$.

\State%
For each $n>0$,
\begin{quote}
for $1 \leq i \leq m$,
\begin{quote}
define a function $\hat{\varphi}^{n}_{i}$ on $\partial D_{i}$ as follows. For $\xi \in \partial D_{i}$, if $\phi_{i} (\xi) \in \partial M$, let $\hat{\varphi}^{n}_{i} (\xi) = \varphi (\phi_{i} (\xi))$. Otherwise, define
\begin{equation}\label{alg_numerical_parallel_1}
\hat{\varphi}^{n}_{i} (\xi) = \sum_{j=1}^{m} \rho_{j} (\phi_{i} (\xi)) \cdot \hat{u}_{j}^{n-1} (\phi_{j}^{-1} \circ \phi_{i} (\xi)).
\end{equation}
Use PINNs method to find a function $\hat{u}^{n}_{i} \in V_{i}$ which approximates the solution to \eqref{eqn_domain_operator} with boundary value $\hat{\varphi}^{n}_{i}$.
\end{quote}
\end{quote}

Define $\hat{u}^{n} = (\hat{u}_{1}^{n}, \cdots, \hat{u}_{m}^{n}) \in V$.

\end{algorithmic}
\end{algorithm}

The sum in \eqref{alg_numerical_parallel_1} needs more explanations. If $\phi_{i} (\xi) \notin \phi_{j} (D_{j})$, then $\phi_{j}^{-1} \circ \phi_{i} (\xi)$ is undefined. However, in this situation, $\rho_{j} (\phi_{i} (\xi)) =0$ and hence the summand $\rho_{j} (\phi_{i} (\xi)) \cdot \hat{u}_{j}^{n-1} (\phi_{j}^{-1} \circ \phi_{i} (\xi))$ is actually redundant. So \eqref{alg_numerical_parallel_1} indeed makes sense. Furthermore, if $\phi_{i} (\xi) \in \phi_{j} (D_{j})$, then
\[
\rho_{j} (\phi_{i} (\xi)) = (\rho_{j} \circ \phi_{j}) (\phi_{j}^{-1} \circ \phi_{i} (\xi)).
\]
Here $\rho_{j} \circ \phi_{j}$ is a function on $D_{j}$ with an explicit formula. Thus $\rho_{j} (\phi_{i} (\xi))$ can be obtained once we get the coordinates of $\phi_{j}^{-1} \circ \phi_{i} (\xi)$.

Like those in Algorithm \ref{alg_numerical_sequential}, each $\hat{u}^{n}_{i}$ in Algorithm \ref{alg_numerical_parallel} is a numerical imitation of the $u^{n}_{i}$ in Algorithm \ref{alg_continuous_parallel}, and each $\hat{u}^{n}_{i} \circ \phi_{i}^{-1}$ is expected to approximate the exact solution $u$ of \eqref{eqn_problem} on $\phi (D_{i})$. However, unlike those in Algorithm \ref{alg_numerical_sequential}, the $\hat{\varphi}^{n}_{i}$'s and $\hat{u}^{n}_{i}$'s for $1 \le i \le m$ in Algorithm \ref{alg_numerical_parallel} can be obtained independently and parallelly at each iterative step $n$.

Similar to Algorithms \ref{alg_continuous_sequential} and \ref{alg_continuous_parallel}, when $M$ is decomposed into two subdomains, Algorithms \ref{alg_numerical_sequential} and \ref{alg_numerical_parallel} are both reduced to Schwarz Alternating Methods. The reason is the same as the one explained in $\S$\ref{subsec_continuous_algorithm}. The following Algorithm \ref{alg_numerical_alternating} is a numerical Schwarz Alternating Method on manifolds.

\begin{algorithm}[ht]
\caption{Numerical Schwarz Alternating Method with PINNs to solve \eqref{eqn_problem}, under Assumption \ref{asp_decomposition_numerical} with $m=2$.}
\label{alg_numerical_alternating}

\begin{algorithmic}[1]
\State%
Choose an initial guess functions $\hat{u}^{0}_{2}$ on $D_{2}$.

\State%
For each $n>0$,
\begin{quote}
define a function $\hat{\varphi}^{n}_{1}$ on $\partial D_{1}$ as follows. For $\xi \in \partial D_{1}$, if $\phi_{1} (\xi) \in \partial M$, let $\hat{\varphi}^{n}_{1} (\xi) = \varphi (\phi_{1} (\xi))$. Otherwise, define
\[
\hat{\varphi}^{n}_{1} (\xi) = \hat{u}^{n-1}_{2} (\phi_{2}^{-1} \circ \phi_{1} (\xi)).
\]
Use PINNs method to find a function $\hat{u}^{n}_{1} \in V_{1}$ which approximates the solution to \eqref{eqn_domain_operator} with boundary value $\hat{\varphi}^{n}_{1}$.

Define a function $\hat{\varphi}^{n}_{2}$ on $\partial D_{2}$ as follows. For $\xi \in \partial D_{2}$, if $\phi_{2} (\xi) \in \partial M$, let $\hat{\varphi}^{n}_{2} (\xi) = \varphi (\phi_{2} (\xi))$. Otherwise, define
\[
\hat{\varphi}^{n}_{2} (\xi) = \hat{u}^{n}_{1} (\phi_{1}^{-1} \circ \phi_{2} (\xi)).
\]
Use PINNs method to find a function $\hat{u}^{n}_{2} \in V_{2}$ which approximates the solution to \eqref{eqn_domain_operator} with boundary value $\hat{\varphi}^{n}_{2}$.
\end{quote}
\end{algorithmic}
\end{algorithm}

Note that PINNs method essentially solves a problem of optimization. Typically, a search for the optimal function of this problem is an iterative process. As a result, the above Algorithms \ref{alg_numerical_sequential}, \ref{alg_numerical_parallel} and \ref{alg_numerical_alternating} are nested iterations. The outer iterations are frameworks provided by DDMs. The inner iterations are the implementations of PINNs at each iterative step of the frameworks.

\subsection{Physics-Informed Neural Networks}
The physics-informed neural networks (PINNs) method popularized by \cite{RPK2019} can solve general (linear or nonlinear) differential equations in domains in Euclidean spaces. It has many advantages. A major advantage is its mesh-free property. Another major one is that it has strong ability to mitigate the curse of dimensionality. In this paper, we shall mainly take the second advantage mentioned here. More precisely, we shall solve \eqref{eqn_domain_operator} with a certain boundary value $\hat{\varphi}^{n}_{i}$ in a domain $D_{i}$ in $\mathbb{R}^{d}$, where $d \ge 5$. In the case of such high dimensions, \eqref{eqn_domain_operator} can hardly be solved effectively by traditional methods, including finite difference method and finite element method.

There are three main ingredients of PINNs: (1) the architecture of neural networks; (2) a loss function to be optimized; (3) an algorithm of optimization. Let's introduce one by one the ingredients in our case.

First, we shall build two types of deep neural networks on $D_{i}$, which will be described in $\S$\ref{sec_experiment}. The activation function is chosen as the hyperbolic tangent function $\tanh$ which is $C^{\infty}$ and even analytic. As before, let $V_{i}$ denote the function space associated with a chosen neural network on $D_{i}$. Now $V_{i} \subset C^{\infty} (D_{i})$.

Second, we define the loss function $\mathcal{L}_{PINN}$ for \eqref{eqn_domain_operator} with boundary value $\hat{\varphi}^{n}_{i}$ as,
\begin{equation}\label{eqn_loss}
\forall v \in V_{i}, \qquad \mathcal{L}_{PINN} (v) = (1-\lambda_{BC}) \mathcal{L}_{PDE} (v) + \lambda_{BC} \mathcal{L}_{BC} (v),
\end{equation}
where $\lambda_{BC} \in (0,1)$ is a positive hyperparameter of penalty. (The abbreviation ``BC" stands for ``boundary condition".) In addition,
\begin{equation}\label{eqn_loss_pde}
\mathcal{L}_{PDE} (v) = \frac{1}{|D_{i}|} \| \hat{L} v - f \circ \phi_{i} \|_{L^{2} (D_{i})}^{2} = \frac{1}{|D_{i}|} \int_{D_{i}} |\hat{L} v(x) - f \circ \phi_{i} (x)|^{2} \mathrm{d} x
\end{equation}
and
\begin{equation}\label{eqn_loss_bc}
\mathcal{L}_{BC} (v) = \frac{1}{|\partial D_{i}|} \| v - \hat{\varphi}^{n}_{i} \||_{L^{2} (\partial D_{i})}^{2} = \frac{1}{|\partial D_{i}|} \int_{\partial D_{i}} |v(y) - \hat{\varphi}^{n}_{i} (y)|^{2} \mathrm{d}y,
\end{equation}
where $|D_{i}|$ (resp.~$|\partial D_{i}|$) is the $d$-dimensional (resp.~$(d-1)$-dimensional) volume of $D_{i}$ (resp.~$\partial D_{i}$). The above integrals are numerically approximated using Monte Carlo method. We shall solve the following optimization problem
\begin{equation}\label{eqn_optimization}
\hat{u}^{n}_{i} := \mathrm{arg} \underset{v \in V_{i}}{\mathrm{min}} \ \mathcal{L}_{PINN} (v).
\end{equation}
At first glance, \eqref{eqn_optimization} is a linear least squares problem. But it is actually not because, as mentioned before, $V_{i}$ is not a linear space. Note that all elements of $V_{i}$ are represented by a fixed finite number, say $N$, of parameters. Thus $V_{i}$ is the image of a continuous map $F \colon \mathbb{R}^{N} \rightarrow C^{\infty} (D_{i})$, where $F$ is determined by the architecture of the neural network and is often highly nonlinear. In this context, each $X \in \mathbb{R}^{N}$ is a collection of parameters with specific values representing a function $F (X) \in V_{i}$. Now \eqref{eqn_optimization} is converted to the following down-to-earth optimization problem
\begin{equation}\label{eqn_optimization_parameter}
\mathrm{arg} \underset{X \in \mathbb{R}^{N}}{\mathrm{min}} \ \mathcal{L}_{PINN} (F (X))
\end{equation}
which is highly nonlinear and nonconvex.

Third, in practice, \eqref{eqn_optimization_parameter} can only be approximately solved. A search of good approximations is usually an iterative process. We shall use Adaptive Moment Estimation (Adam) \cite{kingma_ba} optimizer to implement this search. Furthermore, to compute \eqref{eqn_loss_pde} and \eqref{eqn_loss_bc}, we shall randomly sample a collection of points $\{ x_{i} \}_{i=1}^{N_{p}}$ (resp.~$\{ y_{i} \}_{i=1}^{N_{b}}$) in $D_{i}$ (resp.~on $\partial D_{i}$). These sampling points are updated at each Adam iterative step in order to avoid overfitting, although the numbers $N_{p}$ and $N_{b}$ remain constants.

\section{Numerical Experiments}\label{sec_experiment}
We perform numerical tests of Algorithms \ref{alg_numerical_sequential} and \ref{alg_numerical_parallel} on manifolds with or without boundary. The dimensions vary from $5$ to $10$. The manifolds are $S^{5}$, $S^{10}$, $B^{3} \times S^{3}$ and $\mathbb{CP}^{3}$. Here $S^{d}$ is the unit sphere in $\mathbb{R}^{d+1}$, $B^{d}$ is the unit ball in $\mathbb{R}^{d}$, and $\mathbb{CP}^{3}$ is the complex projective space with complex dimension $3$ and hence real dimension $6$.

\subsection{On \texorpdfstring{$S^{5}$}{S5}}
We shall perform numerical tests on $5$-dimensional manifold $S^{5}$.

First of all, we introduce a general decomposition on $M= S^{d}$, where $S^{d}$ is the unit sphere in $\mathbb{R}^{d+1}$, i.e.
\[
M = S^{d} = \left\{ y= (y_{1}, \dots, y_{d+1}) \in \mathbb{R}^{d+1} \middle| \| y \|^{2} = \sum_{i=1}^{d+1} y_{i}^{2}  = 1  \right\}.
\]
This is a $d$-dimensional manifold without boundary. We decompose $S^{d}$ into two subdomains as follows. By stereographic projections from the south pole $(0, \dots ,0,-1)$ and north pole $(0, \dots, 0,1)$, we obtain two subdomains $M_{1}$ and $M_{2}$ with coordinates whose interiors form an open cover of $S^{d}$.
\begin{figure}[htbp]
  \centering
  \begin{minipage}{0.48\textwidth}
    \centering
    \includegraphics[width=\textwidth]{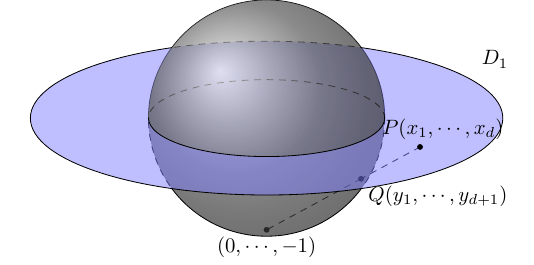}
  \end{minipage}
  \hfill
  \begin{minipage}{0.48\textwidth}
    \centering
    \includegraphics[width=\textwidth]{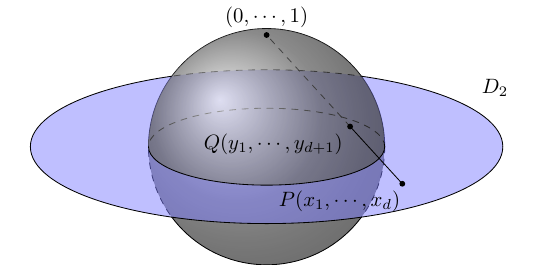}
  \end{minipage}
  \caption{An illustration of stereographic projections.}
  \label{fig_sphere}
\end{figure}
For an illustration please refer to Fig.~\ref{fig_sphere}, where the vertical direction stands for the direction of the $(d+1)$-th coordinate axis, the disk $D_{i}$ is a closed ball in
\[
\mathbb{R}^{d} \cong \mathbb{R}^{d} \times \{ 0 \} \subset \mathbb{R}^{d+1},
\]
and the intersection of $\mathbb{R}^{d} \times \{ 0 \}$ and $S^{d}$ is the equator of $S^{d}$. For each point $P= (x_{1}, \dots, x_{d}) \in D_{1}$, the line segment between $P$ and the south pole $(0, \dots, 0, -1)$ intersects $S^{d}$ at a single point $Q = (y_{1}, \dots, y_{d+1})$ other than $(0, \dots, 0, -1)$. The map $P \mapsto Q$ provides an embedding $\phi_{1}: D_{1} \rightarrow S^{d}$. We obtain another embedding $\phi_{2}$ if the south pole is replaced with the north pole. More precisely, we have
\begin{equation}\label{eqn_Sd_domain}
D_{1} = D_{2} = B^{d}(r) = \{ x= (x_{1}, \dots, x_{d}) \in \mathbb{R}^{d} \mid| \| x \| \le r \} \subset \mathbb{R}^{d},
\end{equation}
i.e.~$D_{1}$ and $D_{2}$ are the ball in $\mathbb{R}^{d}$ with center the origin $\mathbf{0}$ and radius $r$. The stereographic projections provide diffeomorphisms $\phi_{i}: D_{i} \rightarrow M_{i}$ as
\begin{eqnarray}\label{eqn_upper_Sd}
\phi_{1} \colon \ D_{1} & \rightarrow & M_{1} \subset S^{d} \subset \mathbb{R}^{d+1} \nonumber \\
x & \mapsto & \left( \frac{2x}{1 + \|x\|^{2}}, \frac{1 - \|x\|^{2}}{1 + \|x\|^{2}} \right)
\end{eqnarray}
and
\begin{eqnarray}\label{eqn_lower_Sd}
\phi_{2} \colon \ D_{2} & \rightarrow & M_{2} \subset S^{d} \subset \mathbb{R}^{d+1} \nonumber \\
x & \mapsto & \left( \frac{2x}{1 + \|x\|^{2}}, \frac{-1 + \|x\|^{2}}{1 + \|x\|^{2}} \right).
\end{eqnarray}
To guarantee $S^{d} = \bigcup_{i=1}^{2} (M_{i} \setminus \overline{\gamma_{i}}) = \bigcup_{i=1}^{2} (M_{i} \setminus \partial M_{i})$, we have to let $r>1$. The larger $r$ is, the more overlapping there will be. The transitions of coordinates are given by
\[
\phi_{2}^{-1} \circ \phi_{1} (x) = \phi_{1}^{-1} \circ \phi_{2} (x) = \frac{x}{\| x \|^{2}}.
\]

Second, equip $S^{d}$ with the Riemannian metric $g$ inherited from the standard one on $\mathbb{R}^{d+1}$. On each $D_{i}$,
\[
g = 4 (1+ \| x \|^{2})^{-2} \sum_{\alpha = 1}^{d} \mathrm{d} x_{\alpha} \otimes \mathrm{d} x_{\alpha},
\]
and
\[
\Delta v = 4^{-1} (1+ \| x \|^{2})^{2} \sum_{\alpha=1}^{d} \frac{\partial^{2} v}{\partial x_{\alpha}^{2}} - 2^{-1} (d-2) (1+ \| x \|^{2}) \sum_{\alpha=1}^{d} x_{\alpha} \frac{\partial v}{\partial x_{\alpha}}.
\]

Third, consider the model problem \eqref{eqn_problem} on $S^{d}$ with $b>0$. Choose the exact solution to \eqref{eqn_problem} as
\[
u= y_{d+1},
\]
where $y_{d+1}$ is the $(d+1)$-th coordinate of $\mathbb{R}^{d+1}$. Then
\[
f=(d+b) u
\]
in \eqref{eqn_problem}. On $D_{i}$, the exact solution $u$ has the expression
\[
\hat{u}_{1} (x) := u \circ \phi_{1} (x) = \frac{1 - \|x\|^{2}}{1 + \|x\|^{2}} \qquad \text{and} \qquad \hat{u}_{2} (x) := u \circ \phi_{2} (x) = \frac{-1 + \|x\|^{2}}{1 + \|x\|^{2}}.
\]

In the following, we choose $b=1$. Take $r=1.2$ in \eqref{eqn_Sd_domain}. Numerical experiments are performed on $S^{d}$ with $d=5$ and $10$. Since $S^{d}$ is decomposed into two subdomains, Algorithms \ref{alg_numerical_sequential} and \ref{alg_numerical_parallel} are both reduced to Algorithm \ref{alg_numerical_alternating}. We only need to test Algorithm \ref{alg_numerical_alternating} in the current situation.

Our numerical experiments begin with the case $S^{5}$. We employ a fully connected neural network with $4$ hidden layers. The width of each hidden layer is $500$. The activation function is the hyperbolic tangent function $\tanh$. The $\lambda_{BC}$ in \eqref{eqn_loss} is set to $0.8$. For each inner iteration, we sample a total of $2000$ random points, with $1200$ in the interior and $800$ on the boundary. For each outer iterative step, the inner iteration runs $5000$ steps using the Adam optimizer.

Define the relative $L^{2}$-error of $\hat{u}^{n} = (\hat{u}^{n}_{1}, \hat{u}^{n}_{2})$ with respect to the exact solution $\hat{u} = (\hat{u}_{1}, \hat{u}_{2})$ as
\begin{equation}\label{eqn_relative_L2}
\max_{1 \le i \le 2} \left\{ \frac{\| \hat{u}^{n}_{i} - \hat{u}_{i} \|_{L^{2} (D_{i})}}{\| \hat{u}_{i} \|_{L^{2} (D_{i})}} \right\},
\end{equation}
where $\hat{u}^{n}$ is the $n$-th iterative approximation of the outer iteration.

We perform five numerical tests independently. The mean value and standard deviation of the relative $L^{2}$-errors \eqref{eqn_relative_L2} are reported in Table \ref{tab_S5}, where the relative $L^{2}$-errors are calculated by Monte Carlo method.
\begin{table}[ht]
\begin{center}
\begin{tabular}{|c|c|c|c|c|c|c|}
\hline
$n$ & $0$ & $5$ & $10$ & $15$ & $20$ & $100$ \\
\hline
Mean & $1.0559$ & $0.2067$ & $0.0290$ & $4.18e-3$ & $1.62e-3$ & $3.47e-4$ \\
\hline
Std Dev & $0.0452$ & $0.0316$ & $5.53e-3$ & $1.34e-3$ & $6.59e-4$ & $2.77e-5$ \\
\hline
\end{tabular}
\end{center}
\caption{Relative $L^{2}$-errors of $\hat{u}^{n}$ with respect to $\hat{u}$ on $S^{5}$.}
\label{tab_S5}
\end{table}

We also plot the convergence logarithm curve in Fig.~\ref{fig_S5_error}, where the horizontal coordinate is $n$, the vertical coordinate is the mean of the relative $L^{2}$-errors \eqref{eqn_relative_L2}.
\begin{figure}[htbp]
\begin{center}
  \includegraphics[width=0.5\textwidth]{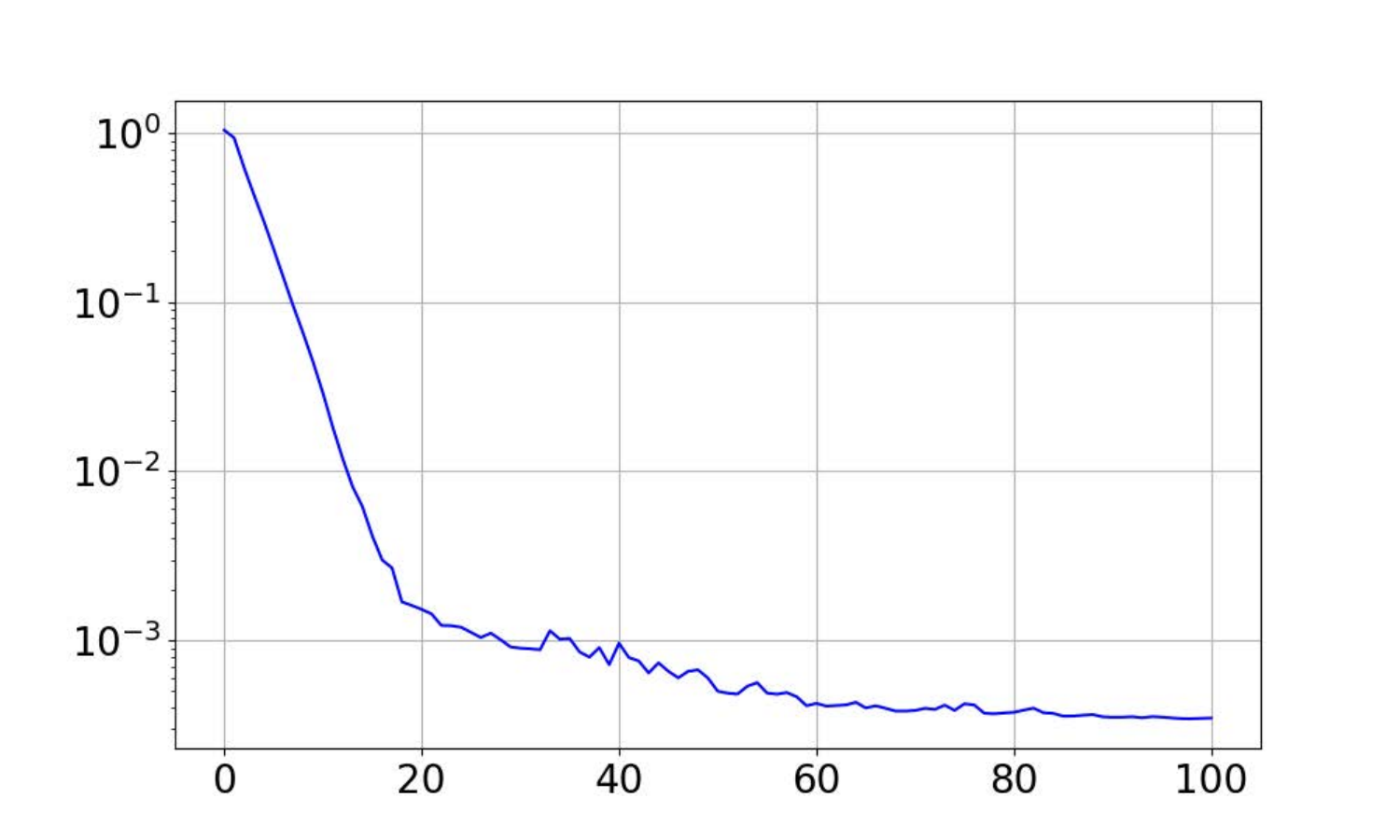}
\end{center}
  \caption{Relative $L^{2}$-convergence of $\hat{u}^{n}$ with respect to $\hat{u}$ on $S^{5}$.}
  \label{fig_S5_error}
\end{figure}

As shown in Table \ref{tab_S5} and Fig.~\ref{fig_S5_error}, the $\hat{u}^{n}$ sequence approximates the exact solution $\hat{u}$ well, and the outer iteration stabilizes after $60$ steps.

We set $\hat{u}^{\infty} := \hat{u}^{100}$ and define the relative $L^{2}$-error of $\hat{u}^{n}$ with respect to $\hat{u}^{\infty}$ as
\begin{equation}\label{eqn_uinfty_relative_L2}
\max_{1 \le i \le 2} \left\{ \frac{\| \hat{u}^{n}_{i} - \hat{u}^{\infty}_{i} \|_{L^{2} (D_{i})}}{\| \hat{u}_{i}^{\infty} \|_{L^{2} (D_{i})}} \right\}.
\end{equation}
The convergence logarithm curve of $\hat{u}^{n}$ with respect to $\hat{u}^{\infty}$ is in Fig.~\ref{fig_S5_infty}, where the horizontal coordinate is $n$, the vertical coordinate is the mean of the relative $L^{2}$-errors \eqref{eqn_uinfty_relative_L2}.
\begin{figure}[htbp]
\begin{center}
  \includegraphics[width=0.5\textwidth]{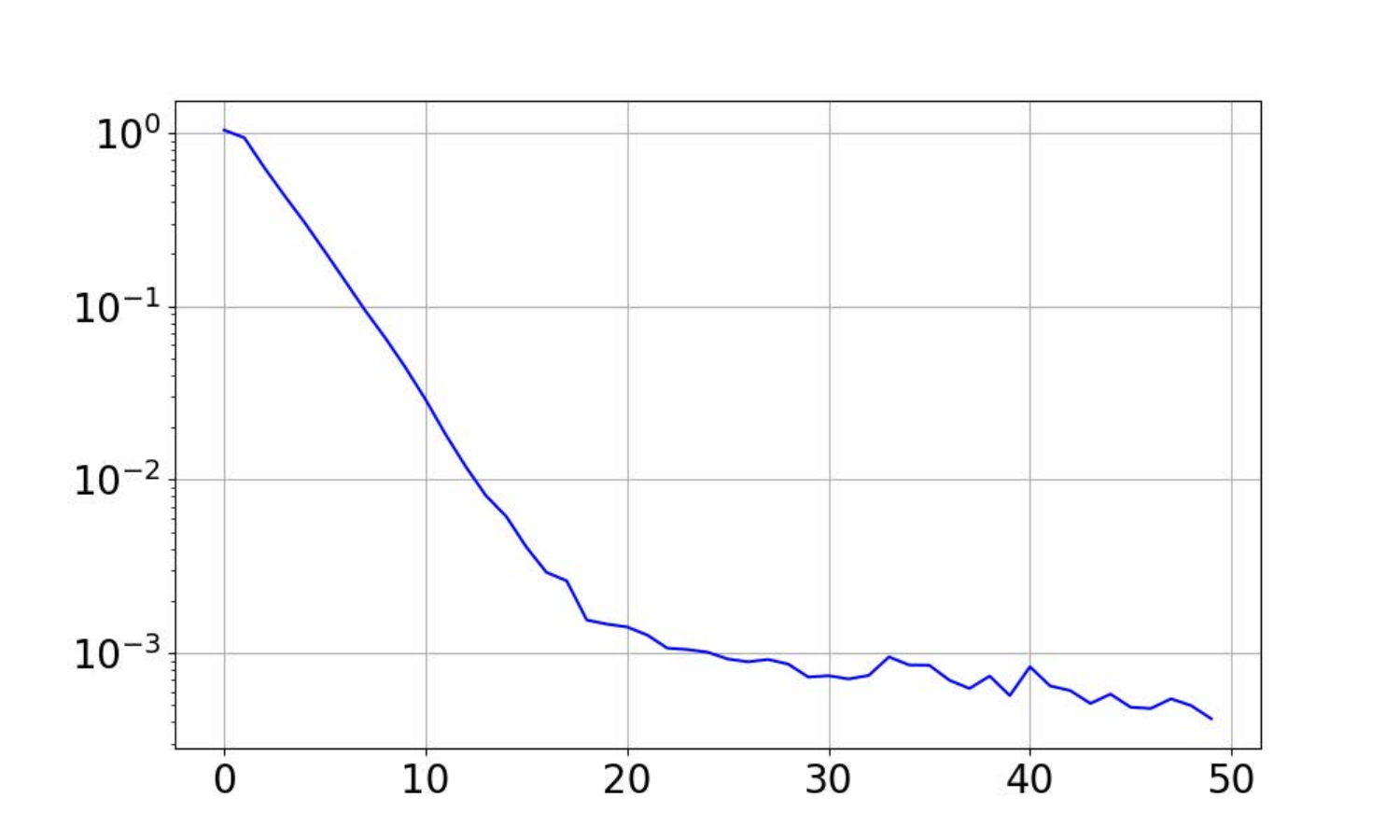}
\end{center}
  \caption{Relative $L^{2}$-convergence of $\hat{u}^{n}$ with respect to $\hat{u}^{\infty}$ on $S^{5}$.}
  \label{fig_S5_infty}
\end{figure}

\subsection{On \texorpdfstring{$S^{10}$}{S10}}
Now we perform numerical experiments on $S^{10}$. The network is a fully connected one with $6$ hidden layers, the width of each hidden layer is $500$. The activation function is $\tanh$. The $\lambda_{BC}$ in \eqref{eqn_loss} is set to $0.9$. For each inner iteration, we sample a total of $5000$ random points, with $4500$ in the interior and $500$ on the boundary. For each outer iterative step, the inner iteration runs $5000$ steps using the Adam optimizer.

We test five times independently. The numerical results are presented in Table \ref{tab_S10} and Figs.~\ref{fig_S10_error} and \ref{fig_S10_infty}. They also show that the the $\hat{u}^{n}$ sequence approximates the exact solution $\hat{u}$ well.
\begin{table}[ht]
\begin{center}
\begin{tabular}{|c|c|c|c|c|c|c|}
\hline
$n$ & $0$ & $5$ & $10$ & $15$ & $20$ & $100$ \\
\hline
Mean & $1.1359$ & $0.5412$ & $0.1435$ & $0.0443$ & $0.0131$ & $1.13e-3$ \\
\hline
Std Dev & $0.1594$ & $0.2394$ & $0.0600$ & $0.0307$ & $9.01e-3$ & $1.09e-4$ \\
\hline
\end{tabular}
\end{center}
\caption{Relative $L^{2}$-errors of $\hat{u}^{n}$ with respect to $\hat{u}$ on $S^{10}$.}
\label{tab_S10}
\end{table}

\begin{figure}[htbp]
\begin{center}
  \includegraphics[width=0.5\textwidth]{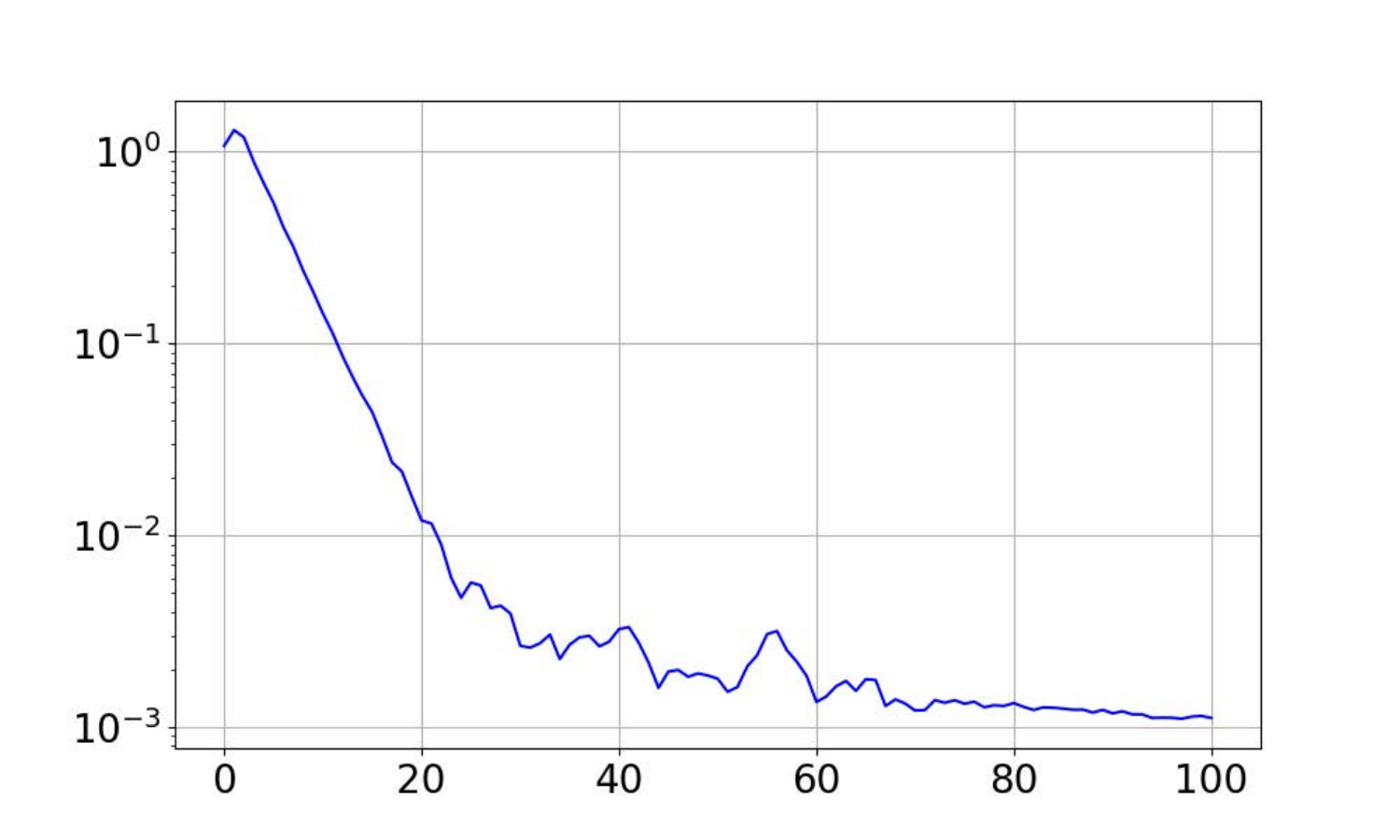}
\end{center}
  \caption{Relative $L^{2}$-convergence of $\hat{u}^{n}$ with respect to $\hat{u}$ on $S^{10}$.}
  \label{fig_S10_error}
\end{figure}

\begin{figure}[htbp]
\begin{center}
  \includegraphics[width=0.5\textwidth]{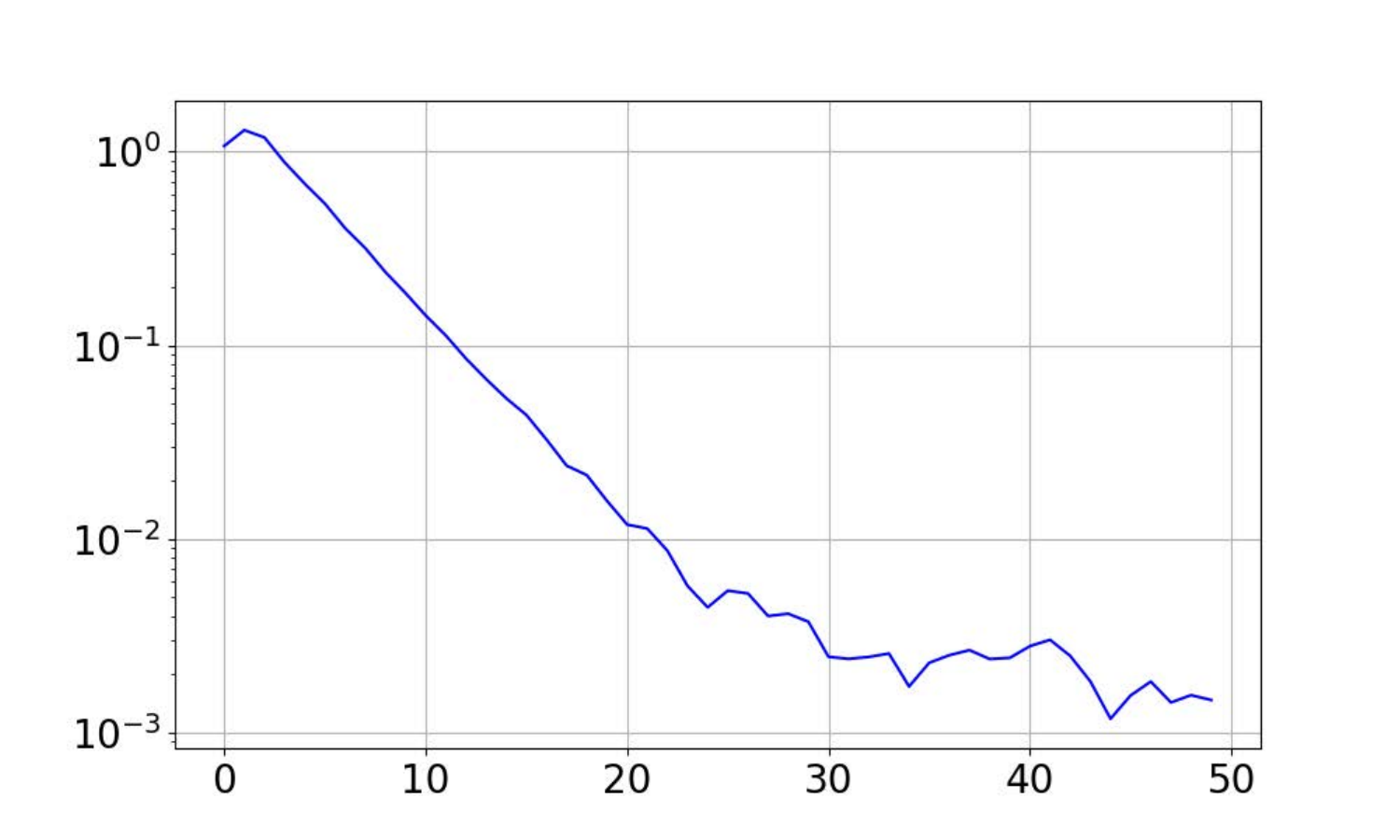}
\end{center}
  \caption{Relative $L^{2}$-convergence of $\hat{u}^{n}$ with respect to $\hat{u}^{\infty}$ on $S^{10}$.}
  \label{fig_S10_infty}
\end{figure}

\subsection{On \texorpdfstring{$B^{3} \times S^{3}$}{B3xS3}}\label{subsec_B3S3}
We shall perform numerical tests on $6$-dimensional manifold $B^{3} \times S^{3}$ which has boundary $S^{2} \times S^{3}$.

First of all, we introduce a general decomposition on $M = B^{p} \times S^{q}$. Here $B^{p}$ is the unit ball in $\mathbb{R}^{p}$, i.e.
\[
B^{p} = \{ y = (y_{1}, \dots, y_{p}) \in \mathbb{R}^{p} \mid \| y \| \le 1 \},
\]
and $S^{q}$ is the unit sphere in $\mathbb{R}^{q+1}$. This $M$ is a manifold with dimension $p+q$ and with boundary $\partial M = (\partial B^{p}) \times S^{q} = S^{p-1} \times S^{q}$.

As mentioned at the end of $\S$\ref{subsec_decompose}, to decompose the product manifold $B^{p} \times S^{q}$, it suffices to decompose its factors $B^{p}$ and $S^{q}$. Since $B^{p}$ is already a simple domain in $B^{p}$, it does not need any further decomposition. We decompose $S^{q}$ by stereographic projections again. Two diffeomorphisms $B^{q} (r) \rightarrow M'_{i} \subset S^{q}$ ($i=1,2$) are obtained, where $M'_{i}$ is a domain in $S^{q}$, and $B^{q} (r)$ is the following ball in $\mathbb{R}^{q}$,
\[
B^{q}(r) = \{ x= (x_{1}, \dots, x_{q}) \in \mathbb{R}^{q} \mid| \| x \| \le r \}.
\]
In addition, the formulas of these two diffeomorphisms are given by \eqref{eqn_upper_Sd} and \eqref{eqn_lower_Sd}.

Overall, we can define
\begin{equation}\label{eqn_BpSq_domain}
D_{1} = D_{2} = B^{p} \times B^{q}(r) \subset \mathbb{R}^{p+q}.
\end{equation}
The maps of decomposition of $B^{p} \times S^{q}$ are two diffeomorphisms
\begin{eqnarray*}
\phi_{1} \colon \ D_{1} = B^{p} \times B^{q}(r) & \rightarrow & M_{1} \subset M \subset B^{p} \times \mathbb{R}^{q+1} \\
(y,x) & \mapsto & \left( y, \frac{2x}{1 + \|x\|^{2}}, \frac{1 - \|x\|^{2}}{1 + \|x\|^{2}} \right)
\end{eqnarray*}
and
\begin{eqnarray*}
\phi_{2} \colon \ D_{2} = B^{p} \times B^{q}(r) & \rightarrow & M_{1} \subset M \subset B^{p} \times \mathbb{R}^{q+1} \\
(y,x) & \mapsto & \left( y, \frac{2x}{1 + \|x\|^{2}}, \frac{-1 + \|x\|^{2}}{1 + \|x\|^{2}} \right),
\end{eqnarray*}
where $y \in B^{p}$ and $x \in B^{q}(r)$. Note that, for $i=1$ and $2$,
\[
\partial D_{i} = \partial [B^{p} \times B^{q}(r)] = [(\partial B^{p}) \times B^{q}(r)] \cup [ B^{p} \times \partial B^{q}(r)].
\]
We also see that
\[
\gamma_{i} = \phi_{i} [(B^{p} \setminus \partial B^{p}) \times \partial B^{q}(r)], \qquad \overline{\gamma_{i}} = \phi_{i} [B^{p} \times \partial B^{q}(r)],
\]
and
\[
M_{i} \cap \partial M = \phi_{i} [(\partial B^{p}) \times B^{q}(r)] = \phi_{i} [S^{p-1} \times B^{q}(r)].
\]
To guarantee $B^{p} \times S^{q} = \bigcup_{i=1}^{2} (M_{i} \setminus \overline{\gamma_{i}})$, we have to let $r>1$. The larger $r$ is, the more overlapping there will be. The transitions of coordinates are given by
\[
\phi_{2}^{-1} \circ \phi_{1} (y,x) = \phi_{1}^{-1} \circ \phi_{2} (y,x) = \left( y, \frac{x}{\| x \|^{2}} \right).
\]

Second, equip $B^{p} \times S^{q}$ with the Riemannian metric $g$ inherited from the standard one on $\mathbb{R}^{p+q+1}$. On each $D_{i}$,
\[
g = \sum_{\alpha = 1}^{p} \mathrm{d} y_{\alpha} \otimes \mathrm{d} y_{\alpha} + 4 (1+ \| x \|^{2})^{-2} \sum_{\alpha = 1}^{q} \mathrm{d} x_{\alpha} \otimes \mathrm{d} x_{\alpha},
\]
and
\[
\Delta v = \sum_{\alpha=1}^{p} \frac{\partial^{2} v}{\partial y_{\alpha}^{2}} + 4^{-1} (1+ \| x \|^{2})^{2} \sum_{\alpha=1}^{q} \frac{\partial^{2} v}{\partial x_{\alpha}^{2}} - 2^{-1} (q-2) (1+ \| x \|^{2}) \sum_{\alpha=1}^{q} x_{\alpha} \frac{\partial v}{\partial x_{\alpha}}.
\]

Third, consider the model problem \eqref{eqn_problem} on $B^{p} \times S^{q}$ with $b \ge 0$. Choose the exact solution to \eqref{eqn_problem} as
\[
u (y,y')= \sin (\pi y_{p}) + y'_{q+1},
\]
where $y= (y_{1}, \dots, y_{p}) \in B^{p}$ and $y' = (y'_{1}, \dots, y'_{q+1}) \in S^{q} \subset \mathbb{R}^{q+1}$. Then
\[
f= (b+ \pi^{2}) \sin (\pi y_{p}) + (b+q) y'_{q+1}
\]
in \eqref{eqn_problem}. On $D_{i}$, $u$ and $f$ has the expression
\[
\hat{u}_{1} (y,x) := u \circ \phi_{1} (y,x) = \sin (\pi y_{p}) + \frac{1 - \|x\|^{2}}{1 + \|x\|^{2}},
\]
\[
\hat{u}_{2} (y,x) :=  u \circ \phi_{2} (y,x) = \sin (\pi y_{p}) + \frac{-1 + \|x\|^{2}}{1 + \|x\|^{2}},
\]
\[
f \circ \phi_{1} (y,x) = (b+ \pi^{2}) \sin (\pi y_{p}) + (b+q) \frac{1 - \|x\|^{2}}{1 + \|x\|^{2}},
\]
and
\[
f \circ \phi_{2} (y,x) = (b+ \pi^{2}) \sin (\pi y_{p}) + (b+q) \frac{-1 + \|x\|^{2}}{1 + \|x\|^{2}}.
\]

Since $B^{p} \times S^{q}$ is decomposed into two subdomains, we also only need to test Algorithm \ref{alg_numerical_alternating}. Now we choose $p=q=3$ and $b=0$. Take $r=1.2$ in \eqref{eqn_BpSq_domain}.

We build a new type of network with $4$ hidden layers, each of width $500$. This network is a variant of ResNet, but it differs considerably from the standard ResNet and closely resembles the one described in \cite[$\S$4]{TKZMZL2024}. The hidden part of our network consists of two successive blocks, each of which has two layers but includes only one shortcut connection. More precisely, a function associated with our network takes the form
\[
F(X) = W_{3} \cdot R_{2} \circ R_{1}(X) + b_{3}
\]
where $W_{3}$ is a weight vector and $b_{3}$ is a bias scalar. For $i=1,2$, the map $R_{i} \colon \mathbb{R}^{n} \rightarrow \mathbb{R}^{m}$ is defined by
\[
R_{i} (X) = \sigma (W_{i,2} \sigma (W_{i,1} X + b_{i,1}) + b_{i,2} + W_{i,0} X),
\]
where $\sigma$ is an activation function, $W_{i,j}$'s are weight matrices, and $b_{i,j}$'s are bias vectors. Each $R_{i} (X)$ corresponds to the map represented by one block of the network. The addition of $W_{i,0} X$ is a shortcut connection. We set $W_{i,0} = I$ when $n=m$; in that case, the shortcut reduces to the identity shortcut used in the standard ResNet.

We take our activation function as $\tanh$. The $\lambda_{BC}$ in \eqref{eqn_loss} is set to $0.995$. For each inner iteration, we sample a total of $2000$ random points, with $200$ in the interior and $1800$ on the boundary. For each outer iterative step, the inner iteration runs $5000$ steps using the Adam optimizer. We perform five tests independently. The numerical results are presented in Table \ref{tab_B3S3} and Figs.~\ref{fig_B3S3_error} and \ref{fig_B3S3_infty}.

\begin{table}[ht]
\begin{center}
\begin{tabular}{|c|c|c|c|c|c|c|}
\hline
$n$ & $0$ & $5$ & $10$ & $15$ & $20$ & $100$ \\
\hline
Mean & $1.0265$ & $2.25e-3$ & $1.62e-3$ & $1.04e-3$ & $7.90e-4$ & $1.69e-4$ \\
\hline
Std Dev & $0.0204$ & $2.91e-4$ & $3.95e-4$ & $1.03e-4$ & $1.02e-4$ & $6.66e-6$ \\
\hline
\end{tabular}
\end{center}
\caption{Relative $L^{2}$-errors of $\hat{u}^{n}$ with respect to $\hat{u}$ on $B^{3} \times S^{3}$.}
\label{tab_B3S3}
\end{table}

\begin{figure}[htbp]
\begin{center}
  \includegraphics[width=0.5\textwidth]{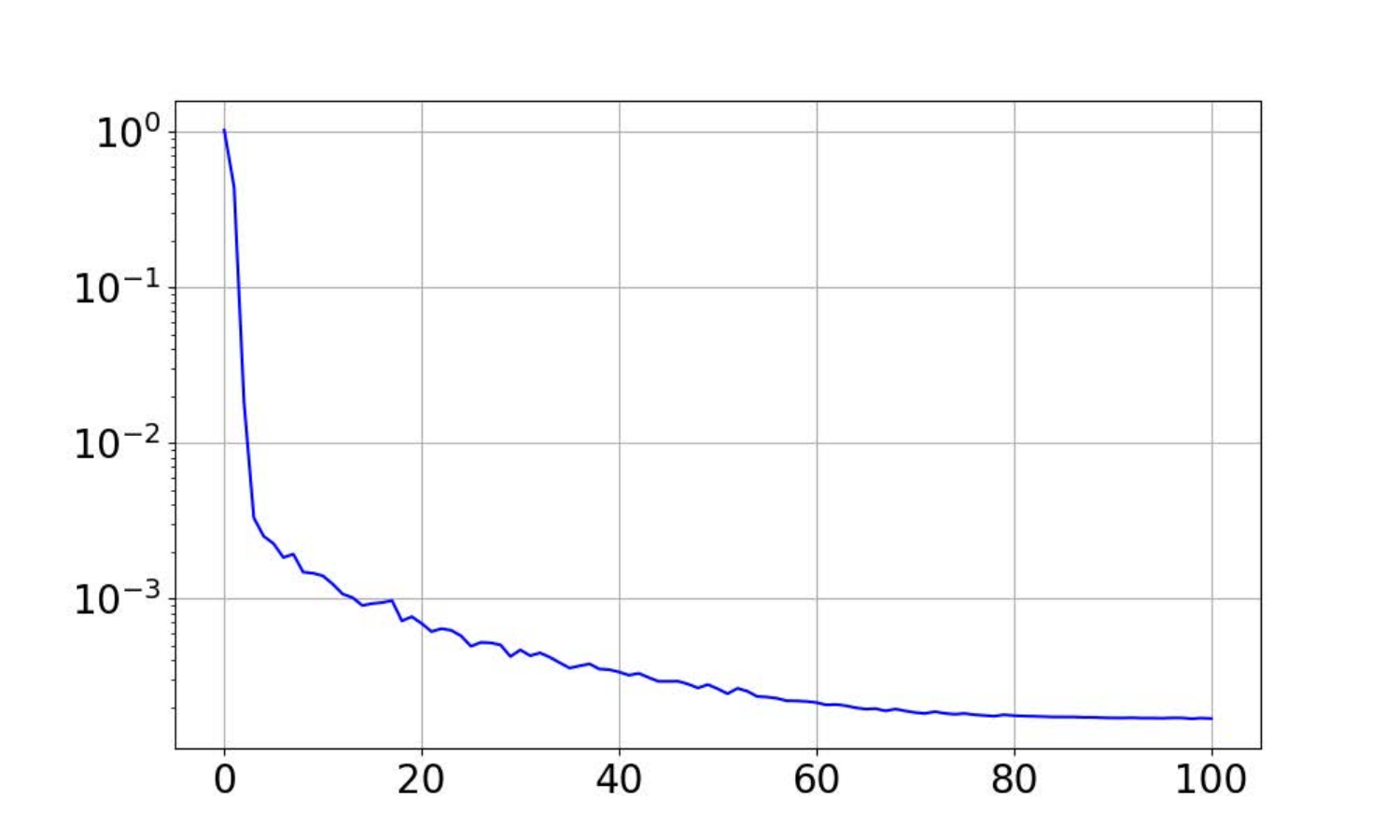}
\end{center}
  \caption{Relative $L^{2}$-convergence of $\hat{u}^{n}$ with respect to $\hat{u}$ on $B^{3} \times S^{3}$.}
  \label{fig_B3S3_error}
\end{figure}

\begin{figure}[htbp]
\begin{center}
  \includegraphics[width=0.5\textwidth]{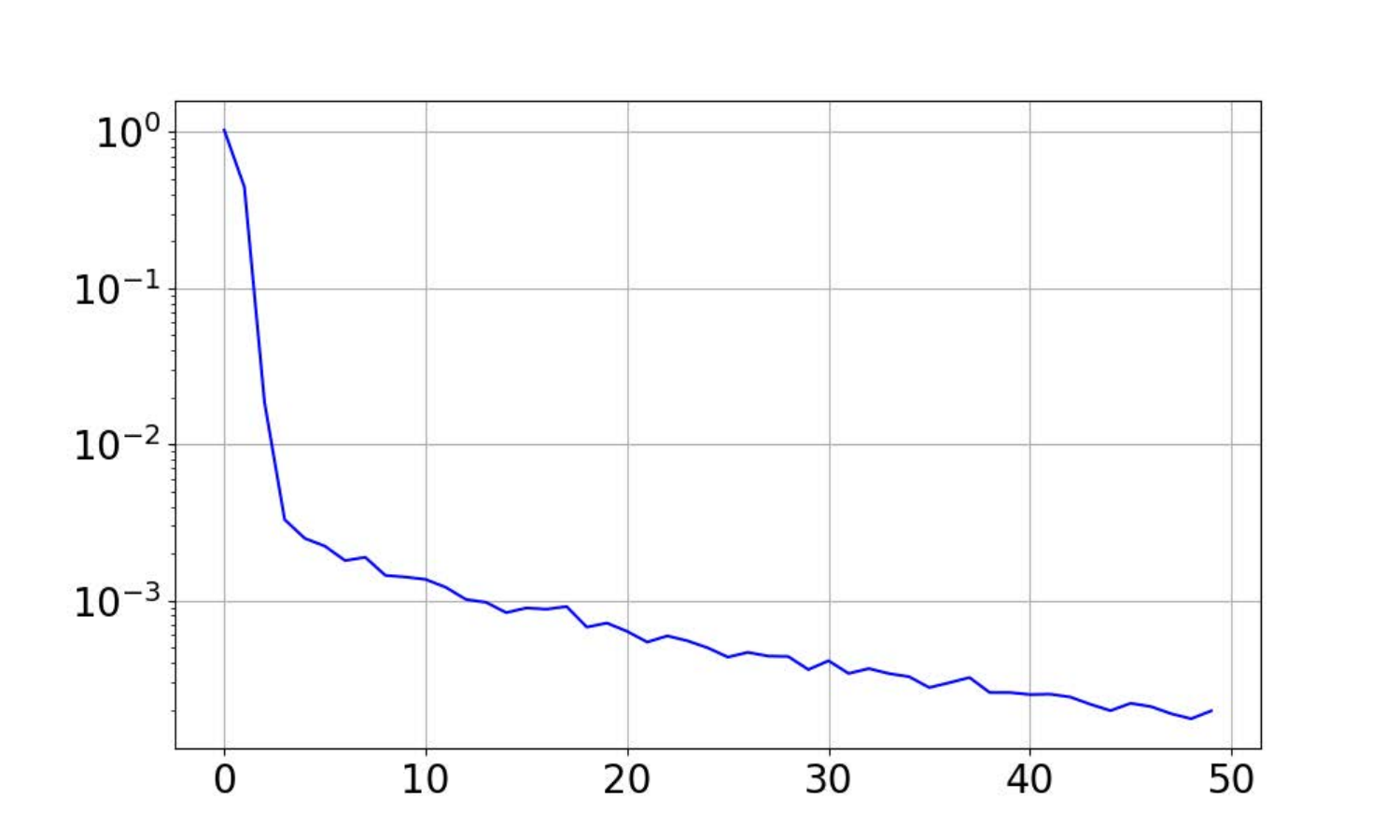}
\end{center}
  \caption{Relative $L^{2}$-convergence of $\hat{u}^{n}$ with respect to $\hat{u}^{\infty}$ on $B^{3} \times S^{3}$.}
  \label{fig_B3S3_infty}
\end{figure}

\subsection{On \texorpdfstring{$\mathbb{CP}^{3}$: Sequential Domain Decomposition Method}{CP3}}
We shall test Algorithm \ref{alg_numerical_sequential} on the complex projective space $\mathbb{CP}^{3}$ which has complex dimension $3$ and real dimension $6$. Unlike the previous examples $S^{d}$ and $B^{p} \times S^{q}$, the complex projective space $\mathbb{CP}^{k}$ is not defined as a submanifold of $\mathbb{R}^{n}$.

First, let's recall the definition of the complex projective space $\mathbb{CP}^{k}$ and formulate a canonical decomposition on it. The $\mathbb{CP}^{k}$ can be defined as a quotient space. Let
\[
\mathbb{C}^{k+1} \setminus \{ \mathbf{0} \} = \{ (w_{0}, w_{1}, \dots, w_{k}) \mid \mathbf{0} \neq (w_{0}, w_{1}, \dots, w_{k}) \in \mathbb{C}^{k+1} \}.
\]
Here $\mathbf{0} \in \mathbb{C}^{k+1}$ is the origin, each $w_{j}$ is a complex number for $0 \leq j \leq k$, and, following the convention of algebraic geometry, the index $j$ starts from $0$ rather than $1$. Define a relation of equivalence on $\mathbb{C}^{k+1} \setminus \{ \mathbf{0} \}$ as
\[
(w_{0}, w_{1}, \dots, w_{k}) \sim (w'_{0}, w'_{1}, \dots, w'_{k})
\]
if and only if
\[
(w_{0}, w_{1}, \dots, w_{k}) = \lambda (w'_{0}, w'_{1}, \dots, w'_{k})
\]
for some $0 \neq \lambda \in \mathbb{C}$. Define
\[
\mathbb{CP}^{k} = (\mathbb{C}^{k+1} \setminus \{ \mathbf{0} \}) / \sim.
\]
Thus, every $P \in \mathbb{CP}^{k}$ can be represented by a vector $(w_{0}, w_{1}, \dots, w_{k}) \in \mathbb{C}^{k+1} \setminus \{ \mathbf{0} \}$. Conventionally, we write
\[
P = [w_{0}, w_{1}, \dots, w_{k}],
\]
where $[w_{0}, w_{1}, \dots, w_{k}]$ are called the \textit{homogeneous coordinates} of $P$. Note that, for $\lambda \neq 0$,
\[
[w_{0}, w_{1}, \dots, w_{k}] = [\lambda w_{0}, \lambda w_{1}, \dots, \lambda w_{k}].
\]
The $\mathbb{CP}^{k}$ is a compact complex manifold without boundary. It has complex dimension $k$ and hence real dimension $2k$. For more details of $\mathbb{CP}^{k}$, see \cite[p.~15]{griffiths_harris}.

Now we decompose $\mathbb{CP}^{k}$ into $k+1$ subdomains $M_{j}$ for $0 \leq j \leq k$. In the following, $z_{j} = x_{j} + \sqrt{-1} y_{j} \in \mathbb{C}$, $x_{j} \in \mathbb{R}$, and $y_{j} \in \mathbb{R}$. We shall identify the complex number $z_{j}$ with the $2$-dimensional real vector $(x_{j}, y_{j})$. Let
\[
B^{2} (r) = \{ z \in \mathbb{C} \mid |z| \le r \} = \{ (x,y) \in \mathbb{R}^{2} \mid \| (x,y) \| \le r \}.
\]
This is a disk in $\mathbb{C} = \mathbb{R}^{2}$ with center $\mathbf{0}$ and radius $r$. For $0 \leq j \leq k$, define
\begin{equation}\label{eqn_CPk_domain}
D_{j} = \prod_{h=0, h \ne j}^{k} B^{2} (r) = \{ (z_{0}, \dots, \widehat{z_{j}}, \dots, z_{k}) \mid \forall h, z_{h} \in B^{2} (r) \} \subset \mathbb{C}^{k} = \mathbb{R}^{2k},
\end{equation}
where $\widehat{z_{j}}$ means $z_{j}$ is missing. Each $D_{j}$ is a polydisk with real dimension $2k$. We have the following $k+1$ diffeomorphisms
\begin{eqnarray*}
\phi_{j} \colon \ D_{j} & \rightarrow & M_{j} \subset  \mathbb{CP}^{k}  \\
(z_{0}, \dots, \widehat{z_{j}}, \dots, z_{k}) & \mapsto & [z_{0}, \dots, z_{j-1}, 1, z_{j+1}, \dots, z_{k}].
\end{eqnarray*}
To guarantee $\mathbb{CP}^{k} = \bigcup_{j=0}^{k} (M_{j} \setminus \partial M_{j})$, we have to let $r>1$. The larger $r$ is, the more overlapping there will be.

The transitions of coordinates are as follows. If $j<h$, then
\[
\phi_{h}^{-1} \circ \phi_{j} (z_{0}, \dots, \widehat{z_{j}}, \dots, z_{k}) = \tfrac{1}{z_{h}} (z_{0}, \dots, z_{j-1}, 1, z_{j+1}, \dots, \widehat{z_{h}}, \dots, z_{k})
\]
and
\[
\phi_{j}^{-1} \circ \phi_{h} (z_{0}, \dots, \widehat{z_{h}}, \dots, z_{k}) = \tfrac{1}{z_{j}} (z_{0}, \dots, \widehat{z_{j}}, \dots, z_{h-1}, 1, z_{h+1}, \dots, z_{k}).
\]

Second, equipping it with the classical \textit{Fubini-Study metric} (c.f.~\cite[p.~30]{griffiths_harris}), $\mathbb{CP}^{k}$ becomes a K\"{a}hler manifold with K\"{a}hler form
\[
\frac{\sqrt{-1}}{2} \partial \bar{\partial} \log \sum_{j=0}^{k} |w_{j}|^{2},
\]
where $[w_{0}, w_{1}, \dots, w_{k}]$ are the homogeneous coordinates of $\mathbb{CP}^{k}$. The \textit{Fubini-Study metric}, denoted by $\mathcal{H}$, is a Hermitian metric. On each $D_{j}$, it is expressed as
\[
\mathcal{H} = (1+ \|z\|^{2})^{-1} \sum_{\alpha = 0, \alpha \neq j}^{k} \mathrm{d} z_{\alpha} \otimes \mathrm{d} \bar{z}_{\alpha}
- (1+ \|z\|^{2})^{-2} \sum_{\alpha = 0, \alpha \neq j}^{k} \sum_{\beta = 0, \beta \neq j}^{k} \bar{z}_{\alpha} z_{\beta} \mathrm{d} z_{\alpha} \otimes \mathrm{d} \bar{z}_{\beta},
\]
where $z = (z_{0}, \dots, \widehat{z_{j}}, \dots, z_{k}) \in D_{j}$,
\[
\|z\|^{2} = \sum_{\alpha = 0, \alpha \neq j}^{k} |z_{\alpha}|^{2} = \sum_{\alpha = 0, \alpha \neq j}^{k} (x_{\alpha}^{2} + y_{\alpha}^{2}).
\]
We choose the Riemannian metric $g$ on $\mathbb{CP}^{k}$ as the real part of $\mathcal{H}$. This $g$ provides the underlying Riemannian structure of the above K\"{a}hler structure. The Laplacian is expressed as
\begin{align*}
\Delta v = & 2 \Delta_{\partial} v = 2 \Delta_{\bar{\partial}} v \\
= & 4 (1+ \|z\|^{2}) \left( \sum_{\alpha = 0, \alpha \neq j}^{k} \frac{\partial^{2} v}{\partial z_{\alpha} \partial \bar{z}_{\alpha}} + \sum_{\alpha = 0, \alpha \neq j}^{k} \sum_{\beta = 0, \beta \neq j}^{k} z_{\alpha} \bar{z}_{\beta} \frac{\partial^{2} v}{\partial z_{\alpha} \partial \bar{z}_{\beta}} \right) \\
= & (1+ \|z\|^{2}) \sum_{\alpha = 0, \alpha \neq j}^{k} \left( \frac{\partial^{2} v}{\partial x_{\alpha}^{2}} + \frac{\partial^{2} v}{\partial y_{\alpha}^{2}} \right) \\
& +(1+ \|z\|^{2}) \sum_{\alpha = 0, \alpha \neq j}^{k} \sum_{\beta = 0, \beta \neq j}^{k} \left[ (x_{\alpha} x_{\beta} + y_{\alpha} y_{\beta}) \left( \frac{\partial^{2} v}{\partial x_{\alpha} \partial x_{\beta}} + \frac{\partial^{2} v}{\partial y_{\alpha} \partial y_{\beta}} \right) \right. \\
& \left. + (x_{\alpha} y_{\beta} - y_{\alpha} x_{\beta}) \left( \frac{\partial^{2} v}{\partial x_{\alpha} \partial y_{\beta}} - \frac{\partial^{2} v}{\partial y_{\alpha} \partial x_{\beta}} \right) \right].
\end{align*}

Third, we consider the model problem \eqref{eqn_problem} with $b>0$. Choose constants $a_{j} \in \mathbb{R}$, $0 \leq j \leq k$. Choose the exact solution to \eqref{eqn_problem} as
\begin{equation}\label{eqn_u_cpk}
u([w_{0}, w_{1}, \dots, w_{k}]) = \sum_{j=0}^{k} a_{j} |w_{j}|^{2},
\end{equation}
where $[w_{0}, \dots, w_{k}]$ are homogeneous coordinates with normalization $\sum_{j=0}^{k} |w_{j}|^{2} =1$. It is easy to see that $u$ is well-defined. The $f$ in \eqref{eqn_problem} is then
\[
f = (4k+4+b) u - 4 \sum_{j=0}^{k} a_{j}.
\]
On $D_{j}$, the exact solution $u$ has the expression
\[
\hat{u}_{j} := u \circ \phi_{j} (z_{0}, \dots, \widehat{z_{j}}, \dots, z_{k}) = \frac{a_{j} + \sum_{\beta = 0, \beta \neq j}^{k} a_{\beta} |z_{\beta}|^{2}}{1 + \sum_{\beta = 0, \beta \neq j}^{k} |z_{\beta}|^{2}}.
\]

We perform numerical test on $\mathbb{CP}^{3}$. Choose $b$ in \eqref{eqn_problem} as $4$, choose  in \eqref{eqn_u_cpk}
\[
(a_{0}, a_{1}, a_{2}, a_{3}) = (1,2,-1,-2).
\]
Take $r=1.2$ in \eqref{eqn_CPk_domain}. Since $\mathbb{CP}^{3}$ is decomposed into $4$ subdomains, Algorithms \ref{alg_numerical_sequential} and \ref{alg_numerical_parallel} are essentially different in this case. Neither of them is reduced to Algorithm \ref{alg_numerical_alternating}. We apply Algorithm \ref{alg_numerical_sequential} first.

The network architecture is identical to that in $\S$\ref{subsec_B3S3}. The $\lambda_{BC}$ in \eqref{eqn_loss} is set to $0.997$. For each inner iteration, we sample a total of $2000$ random points, with $1000$ in the interior and $1000$ on the boundary. For each outer iterative step, the inner iteration runs $5000$ steps using the Adam optimizer. We perform five tests independently. The numerical results are presented in Table \ref{tab_CP3_s} and Figs.~\ref{fig_CP3_s_error} and \ref{fig_CP3_s_infty}. They show the $\hat{u}^{n}$ sequence approximates the exact solution well.

\begin{table}[ht]
\begin{center}
\begin{tabular}{|c|c|c|c|c|c|c|}
\hline
$n$ & $0$ & $5$ & $10$ & $15$ & $20$ & $100$ \\
\hline
Mean & $1.2793$ & $0.1563$ & $0.0508$ & $0.0170$ & $6.32e-3$ & $3.19e-4$ \\
\hline
Std Dev & $0.1111$ & $0.0565$ & $0.0181$ & $6.34e-3$ & $1.97e-3$ & $2.69e-5$ \\
\hline
\end{tabular}
\end{center}
\caption{Relative $L^{2}$-errors of $\hat{u}^{n}$ with respect to $\hat{u}$ on $\mathbb{CP}^{3}$ (Sequential DDM).}
\label{tab_CP3_s}
\end{table}

\begin{figure}[htbp]
\begin{center}
  \includegraphics[width=0.5\textwidth]{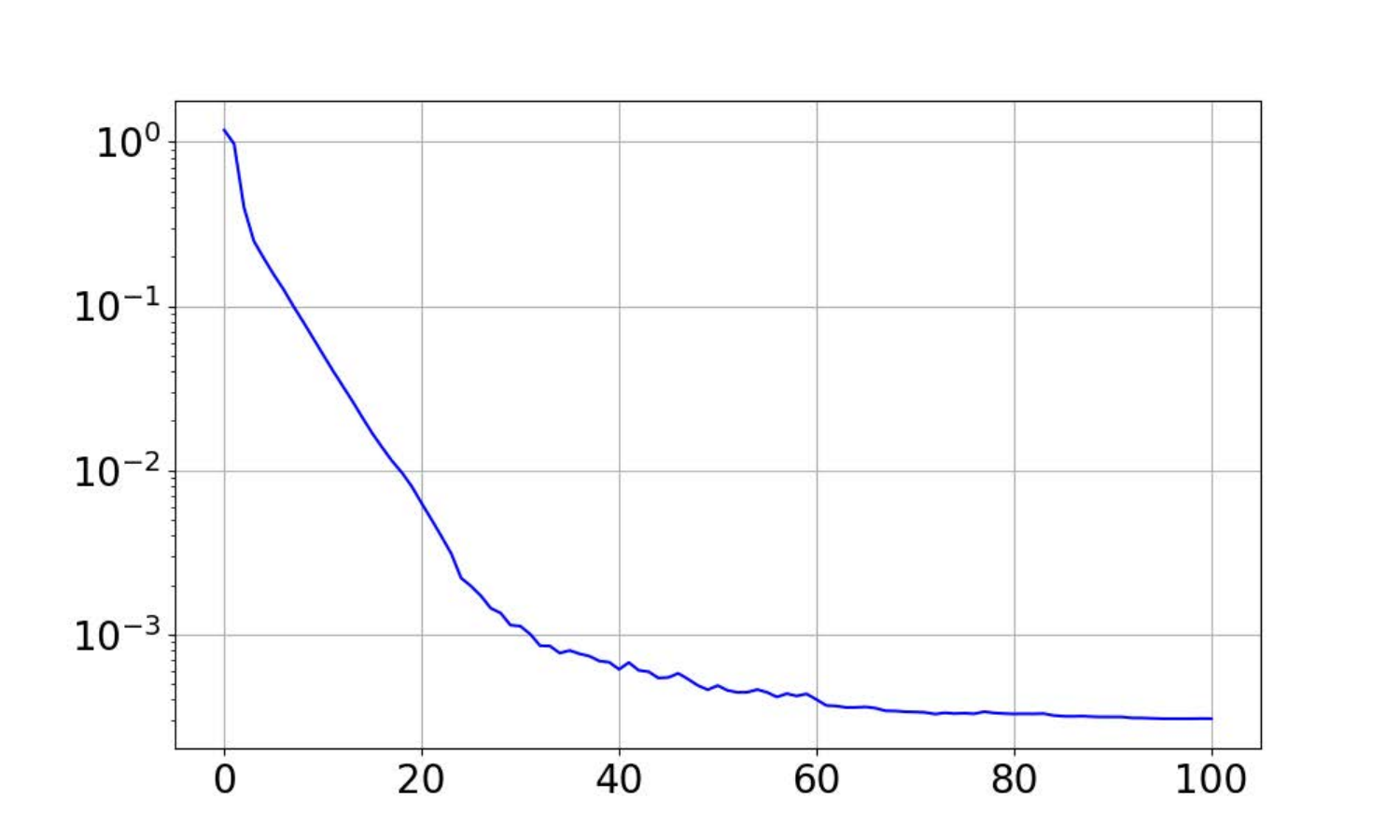}
\end{center}
  \caption{Relative $L^{2}$-convergence of $\hat{u}^{n}$ with respect to $\hat{u}$ on $\mathbb{CP}^{3}$ (Sequential DDM).}
  \label{fig_CP3_s_error}
\end{figure}

\begin{figure}[htbp]
\begin{center}
  \includegraphics[width=0.5\textwidth]{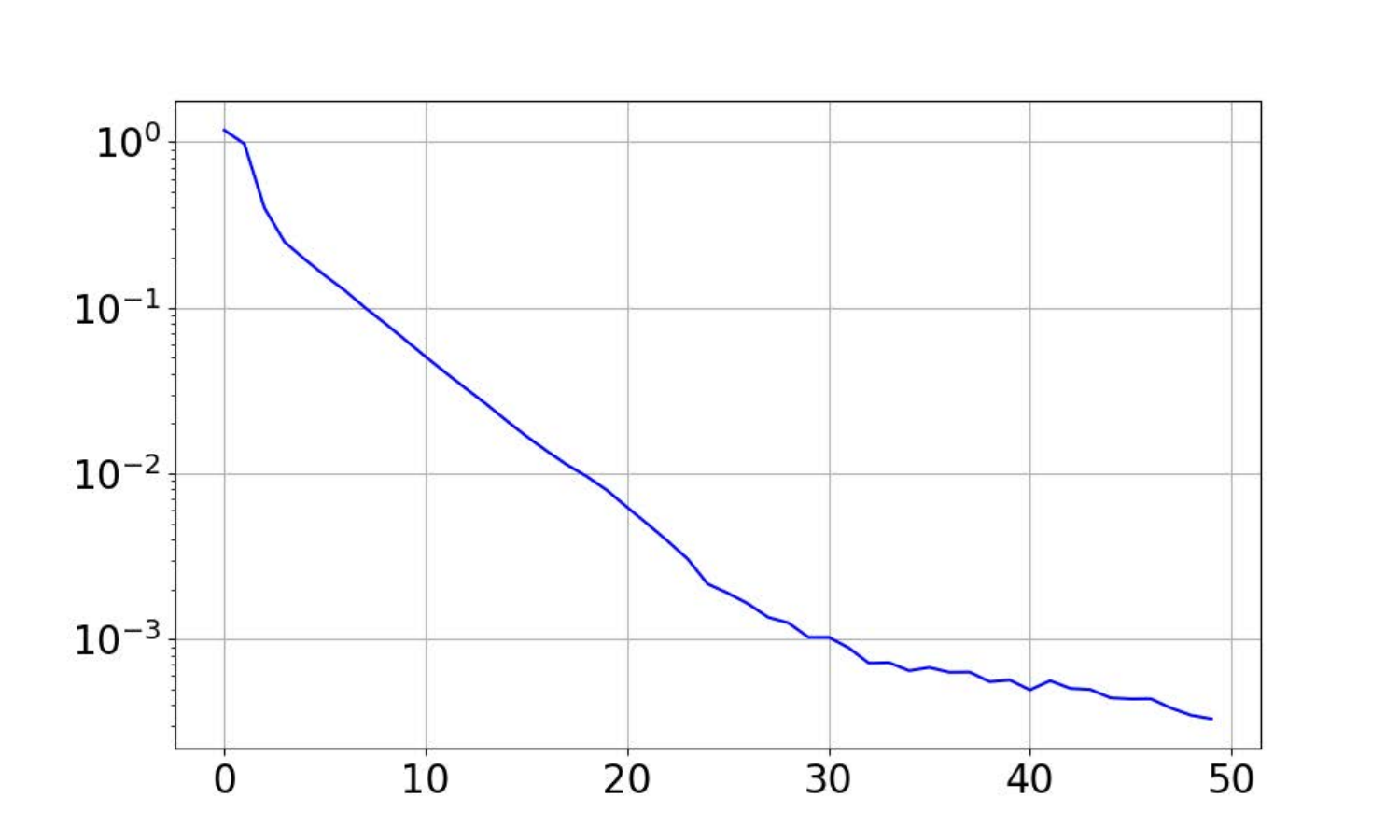}
\end{center}
  \caption{Relative $L^{2}$-convergence of $\hat{u}^{n}$ with respect to $\hat{u}^{\infty}$ on $\mathbb{CP}^{3}$ (Sequential DDM).}
  \label{fig_CP3_s_infty}
\end{figure}

\subsection{On \texorpdfstring{$\mathbb{CP}^{3}$: Parallel Domain Decomposition Method}{CP3}}
Finally, we apply Algorithm \ref{alg_numerical_parallel} to the above specific problem on $\mathbb{CP}^{3}$. Unlike Algorithm \ref{alg_numerical_sequential}, we now need a partition of unity $\{ \rho_{j} \mid 0 \leq j \leq 3 \}$ satisfying Assumption \ref{asp_partition_numerical}.

We first construct a general $\{ \rho_{j} \mid 0 \leq j \leq k \}$ satisfying Assumption \ref{asp_partition_numerical} on $\mathbb{CP}^{k}$. As mentioned before, it suffices to define the $\mu_{j}$ in \eqref{eqn_rho}, which is in turn reduced to the definition of $\mu_{j} \circ \phi_{j}$ on $D_{j}$. Choose $r' \in (1,r)$. Define
\[
\mu_{j} \circ \phi_{j} (z_{0}, \dots, \widehat{z_{j}}, \dots, z_{k}) =
\begin{cases}
0, & \max \{ |z_{h}| \mid 0 \leq h \leq k \} > r'; \\
\prod_{h=0,h \ne j}^{k} (1- (\tfrac{|z_{h|}}{r'})^{2})^{3}, & \text{otherwise}.
\end{cases}
\]
We see that $\mathrm{supp} \mu_{j} \subset M_{j} \setminus \partial M_{j}$, $\mu_{j}$ is $C^{2}$, and $\sum_{j=0}^{k} \mu_{j} >0$. Hence the resulted family of $\rho_{j}$'s are $C^{2}$ functions satisfying Assumption \ref{asp_partition_numerical}.

Let $r' = 0.1 r + 0.9$ and $k=3$. This yields a desired partition of unity. We keep the above settings for the network architecture, hyperparameters, sampling points, inner iterations, and so forth. The experiment is again repeated independently five times. The numerical results are presented in Table \ref{tab_CP3_p} and Figs.~\ref{fig_CP3_p_error} and \ref{fig_CP3_p_infty}.

\begin{table}[ht]
\begin{center}
\begin{tabular}{|c|c|c|c|c|c|c|}
\hline
$n$ & $0$ & $5$ & $10$ & $15$ & $20$ & $100$ \\
\hline
Mean & $1.2613$ & $0.1367$ & $0.0439$ & $0.0236$ & $0.0138$ & $3.21e-4$ \\
\hline
Std Dev & $0.0724$ & $0.0274$ & $0.0236$ & $0.0142$ & $8.53e-3$ & $1.53e-5$ \\
\hline
\end{tabular}
\end{center}
\caption{Relative $L^{2}$-errors of $\hat{u}^{n}$ with respect to $\hat{u}$ on $\mathbb{CP}^{3}$ (Parallel DDM).}
\label{tab_CP3_p}
\end{table}

\begin{figure}[htbp]
\begin{center}
  \includegraphics[width=0.5\textwidth]{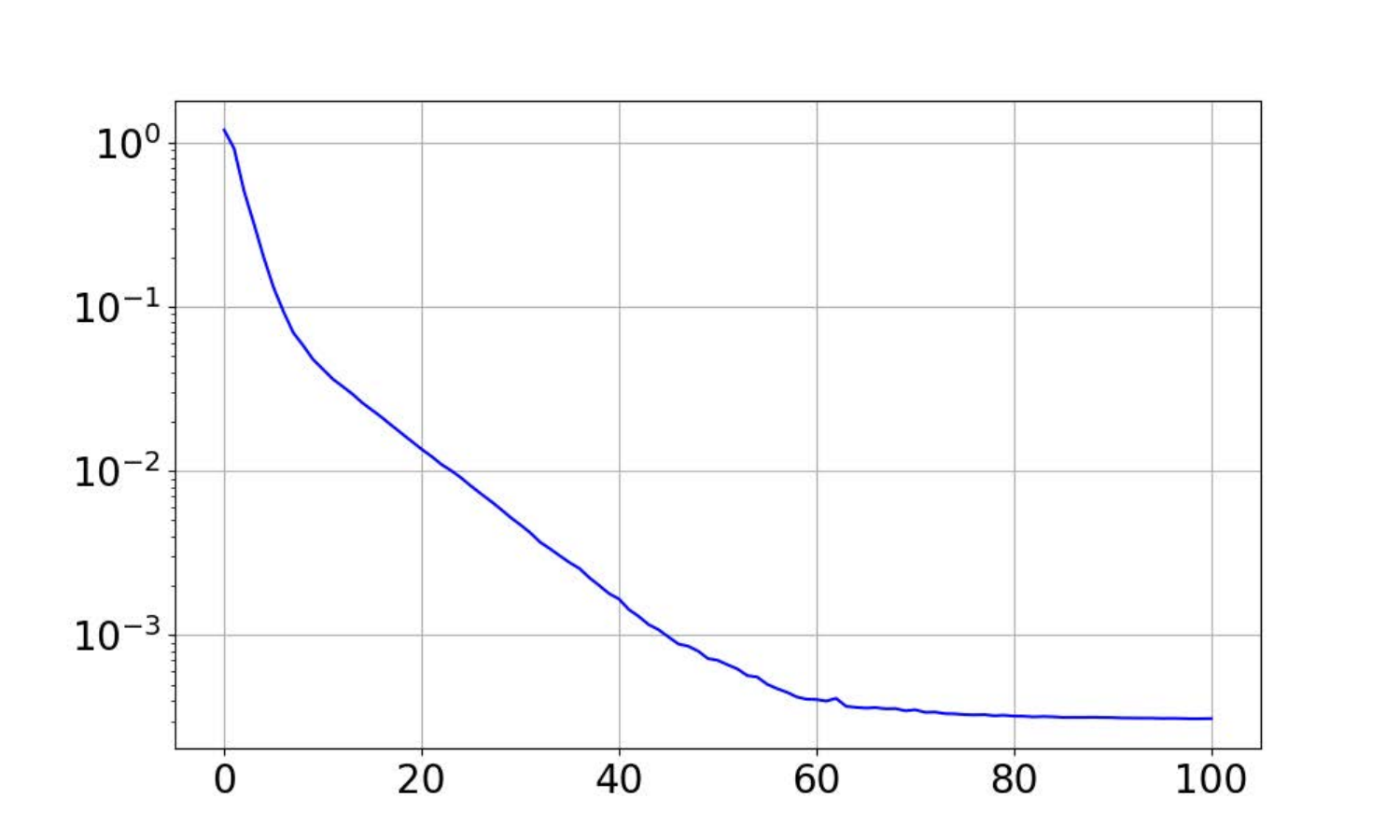}
\end{center}
  \caption{Relative $L^{2}$-convergence of $\hat{u}^{n}$ with respect to $\hat{u}$ on $\mathbb{CP}^{3}$ (Parallel DDM).}
  \label{fig_CP3_p_error}
\end{figure}

\begin{figure}[htbp]
\begin{center}
  \includegraphics[width=0.5\textwidth]{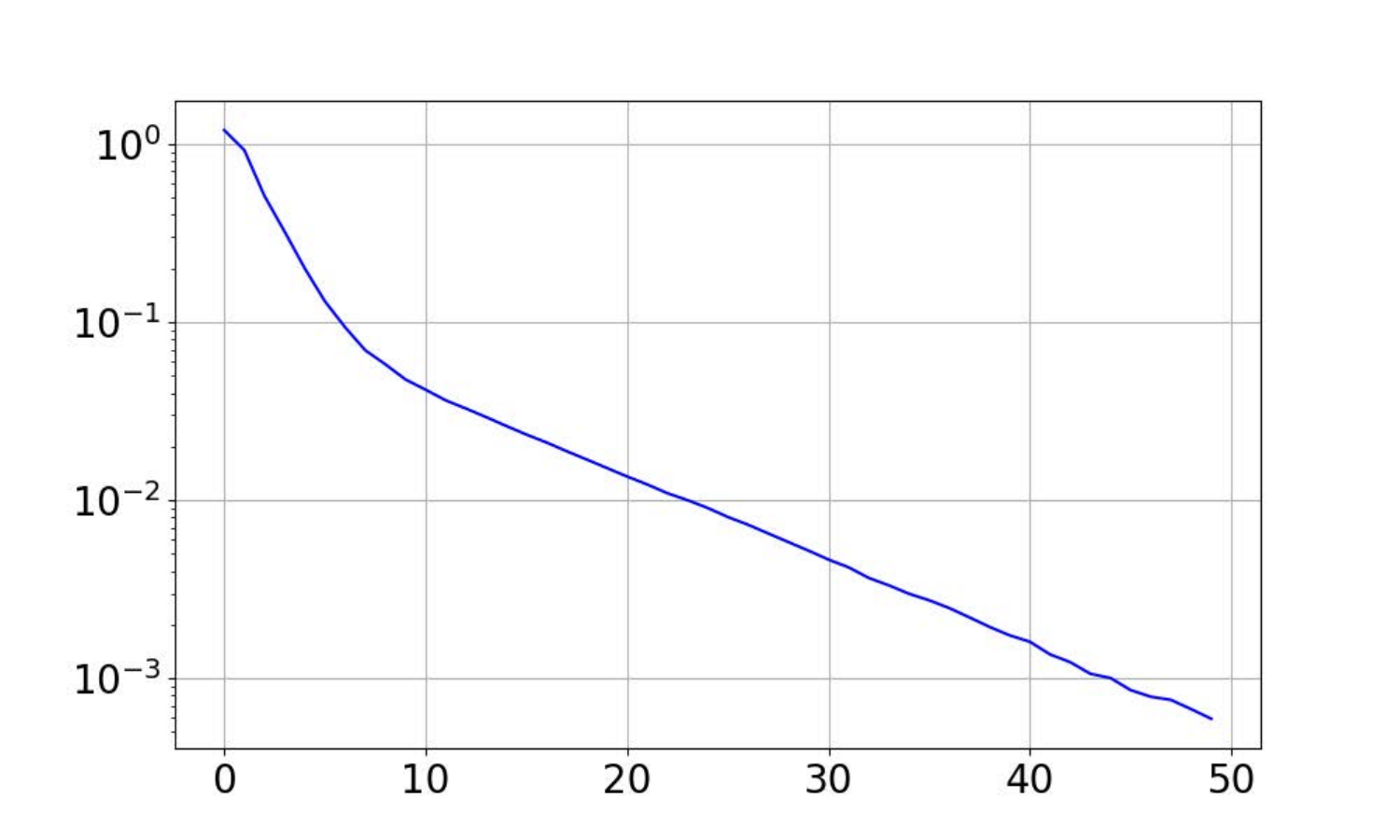}
\end{center}
  \caption{Relative $L^{2}$-convergence of $\hat{u}^{n}$ with respect to $\hat{u}^{\infty}$ on $\mathbb{CP}^{3}$ (Parallel DDM).}
  \label{fig_CP3_p_infty}
\end{figure}

In comparison, the parallel DDM solves the equation on $\mathbb{CP}^{3}$ with the same accuracy as the sequential DDM; however, the outer iteration of the sequential DDM converges faster than that of the parallel DDM.

\section*{Acknowledgements}
We thank Jingfei Chen, Jingrun Chen and Tao Zhou for various discussions. The second author was partially supported by NSFC 11871272. The third author was partially supported by NSFC 12371369.

\end{document}